\documentclass[12pt]{amsart}
\usepackage{tikz-cd}
\usepackage{amssymb}
\usepackage{amsmath,amscd}
\usepackage{mathrsfs}
\usepackage{url}
\usepackage{cite}
\usepackage{fullpage}
\usepackage[hidelinks]{hyperref}
\usepackage{verbatim}
\usepackage{comment}
\usepackage{fancyvrb}
\usepackage{fvextra}
\usepackage{caption}
\usepackage{tikz}
\usepackage{tkz-graph}
\usetikzlibrary{matrix}
 \usetikzlibrary{shapes}
\usetikzlibrary{arrows,decorations.markings}
\usepackage{tikz-cd}
\usepackage{stmaryrd}

\pgfarrowsdeclare{bad to}{bad to}
{
  \pgfarrowsleftextend{-2\pgflinewidth}
  \pgfarrowsrightextend{\pgflinewidth}
}
{
  \pgfsetlinewidth{0.8\pgflinewidth}
  \pgfsetdash{}{0pt}
  \pgfsetroundcap
  \pgfsetroundjoin
  \pgfpathmoveto{\pgfpoint{-3\pgflinewidth}{4\pgflinewidth}}
  \pgfpathcurveto
  {\pgfpoint{-2.75\pgflinewidth}{2.5\pgflinewidth}}
  {\pgfpoint{0pt}{0.25\pgflinewidth}}
  {\pgfpoint{0.75\pgflinewidth}{0pt}}
  \pgfpathcurveto
  {\pgfpoint{0pt}{-0.25\pgflinewidth}}
  {\pgfpoint{-2.75\pgflinewidth}{-2.5\pgflinewidth}}
  {\pgfpoint{-3\pgflinewidth}{-4\pgflinewidth}}
  \pgfusepathqstroke
}

\newtheorem{thm}{Theorem}[section]
\newtheorem{prop}[thm]{Proposition}
\newtheorem{lem}[thm]{Lemma}
\newtheorem{cor}[thm]{Corollary}

\theoremstyle{definition}

\newtheorem{rem}[thm]{Remark}

\numberwithin{equation}{section}
\newcommand{\semis}{\mathrm{ss}}
\newcommand{\s}{\mathsf{S}}
\renewcommand{\l}{\mathsf{L}}

\newcommand{\V}{\mathbb{V}}

\newcommand{\zz}{\mathbb{Z}}
\newcommand{\cc}{\mathbb{C}}
\newcommand{\qq}{\mathbb{Q}}

\renewcommand{\ss}{\mathbb{S}}

\newcommand{\M}{\mathcal{M}}

\renewcommand{\tilde}{\widetilde}

\DeclareMathOperator{\Aut}{Aut}

\DeclareMathOperator{\SL}{SL}

\DeclareMathOperator{\Sym}{Sym}

\DeclareMathOperator{\Sp}{Sp}

\DeclareMathOperator{\Gal}{Gal}

\newcommand{\lstw}{\mathsf{L}\mathsf{S}_{12}}

\usepackage{wrapfig}

\renewcommand{\bar}{\overline}
\newcommand{\Mb}{\overline{\M}}

\usepackage{ytableau,tikz,varwidth, tikz-cd}
\usepackage{hyperref}

\usetikzlibrary{calc, decorations, positioning}
\usepackage[enableskew]{youngtab}
\usepackage{enumerate}
\usepackage{amssymb, mathscinet}

\usepackage[all,cmtip]{xy}

\usepackage{multicol}
\usepackage{enumitem}

\usetikzlibrary{matrix}
\usetikzlibrary{shapes}
\usetikzlibrary{arrows, automata}
\usetikzlibrary{decorations,decorations.pathmorphing,decorations.pathreplacing,decorations.markings}
\usetikzlibrary{through}
\usetikzlibrary{decorations.pathreplacing,calligraphy}

\tikzset{every picture/.style={baseline=-.65ex} }
\tikzset{ext/.style={circle, draw,inner sep=1pt},int/.style={circle,draw,fill,inner sep=1pt},nil/.style={inner sep=1pt}}

\tikzset{every loop/.style={draw}}
\tikzset{
  crossed/.style={
    decoration={markings,mark=at position .5 with {\arrow{|}}},
    postaction={decorate},
    shorten >=0.4pt}}

  \tikzset{->-/.style={decoration={
    markings,
    mark=at position .5 with {\arrow{>}}},postaction={decorate}}}

\newcommand{\ZZ}{\mathbb{Z}}
\newcommand{\CC}{\mathbb{C}}

\newcommand{\Q}{\mathbb{Q}}

\newcommand{\cM}{\mathcal{M}}

\newcommand{\sgn}{\operatorname{sgn}}

\newcommand{\MM}{\overline{\mathcal{M}}}
\newcommand{\RR}{\overline{\mathcal{R}}}

\newcommand{\bbS}{\mathbb{S}}

\newcommand{\bGK}{\overline{\GK}{}}
\newcommand{\hide}[1]{}

\newtheorem{Fact}[thm]{Fact}

\newtheorem{Theorem}[thm]{Theorem}

\newtheorem{Lemma}[thm]{Lemma}

\newtheorem{Remark}[thm]{Remark}

\newtheorem{Proposition}[thm]{Proposition}

\newtheorem{Corollary}[thm]{Corollary}

\newtheorem{Conjecture}[thm]{Conjecture}

  \theoremstyle{definition} 

\theoremstyle{remark}

\newcommand{\gr}{\mathrm{gr}}
\newcommand{\GC}{\mathsf{GC}}

\newcommand{\GK}{\mathsf{GK}}

\newcommand{\omis}{ \begin{tikzpicture}
    \node (v) at (0,0) {$\omega$};
    \node (w) at (1,0) {$1$};
    \draw (v) edge (w);
  \end{tikzpicture}
  \cdots 
  \begin{tikzpicture}
    \node (v) at (0,0) {$\omega$};
    \node (w) at (1,0) {$n$};
    \draw (v) edge (w);
  \end{tikzpicture}}
\newcommand{\tripod}{
\begin{tikzpicture}
    \node[int] (i) at (0,.5) {};
    \node (v1) at (-.5,-.2) {$\omega$};
    \node (v2) at (0,-.2) {$\omega$};
    \node (v3) at (.5,-.2) {$\omega$};
  \draw (i) edge (v1) edge (v2) edge (v3);
  \end{tikzpicture}
}
\newcommand{\tripods}{
\tripod 
\cdots 
\tripod
}

\usepackage{color, colortbl}
\usepackage{array, adjustbox}

\definecolor{Gray}{gray}{0.9}

\newcolumntype{g}{>{\columncolor{Gray}}c}
\newcolumntype{M}{V{6cm}} 
\newcolumntype{N}{V{12cm}} 

\author{Samir Canning}
\author{Hannah Larson}
\author{Sam Payne}
\author{Thomas Willwacher}
\thanks{
S.C. was supported by a Hermann-Weyl-Instructorship from the Forschungsinstitut für Mathematik at ETH Z\"urich.
This research was partially conducted during the period H.L served as a Clay Research Fellow. 
S.P. was supported in part by NSF grants DMS--2053261 and DMS--2302475, and he conducted parts of this research during a visit to the Institute for Advanced Study supported by the Charles Simonyi Endowment. T.W. has been supported by the NCCR SwissMAP, funded by the Swiss National Science Foundation
}

\title{The motivic structures $\mathsf{LS}_{12}$ and $\mathsf{S}_{16}$ in the cohomology of moduli spaces of curves}

\begin{document}

\begin{abstract}
    We study the appearances of  $\mathsf{LS}_{12}$ and $\mathsf{S}_{16}$ in the weight-graded compactly supported cohomology of moduli spaces of curves. As applications, we prove new nonvanishing results for the middle cohomology groups of $\M_9$ and $\M_{11}$ and give evidence to support the conjecture that $\dim_\Q H^{2g + k}_c(\M_g)$ grows at least exponentially with $g$ for almost all $k$. 
\end{abstract}

\maketitle


\section{Introduction}

There are few isomorphism classes of  irreducible Hodge structures and $\ell$-adic Galois representations that appear as subquotients of the compactly supported cohomology groups of moduli spaces of smooth and stable curves in motivic weights less than or equal to 15, and these are completely classified in \cite{BergstromFaberPayne, CLP-STE}, confirming predictions from arithmetic \cite{ChenevierLannes}. It is therefore natural to study these cohomology groups by computing the multiplicities with which each of these irreducible motivic structures appear. 

The strategy of studying the cohomology of moduli spaces one motivic structure at a time has already proved fruitful; appearances of the trivial motivic structure in weight zero refuted a longstanding conjecture about the cohomology of $\M_g$ and established new connections to graph cohomology and Grothendieck-Teichm\"uller theory \cite{CGP21, CGP22}.  Graph complex techniques were then used to study multiplicities of the structures $\l$ and $\s_{12}$ in weights 2 and 11, respectively, providing many nontrivial new lower bounds on dimensions of cohomology groups of $\M_g$ \cite{PayneWillwacher21, PayneWillwacher24}. Further computations of Euler characteristics (alternating sums of multiplicities) of $\s_{12}$ and $\lstw$ in weights 11 and 13 were used to prove that $\M_g$ does not have polynomial point count for $g \geq 9$ \cite{CLPW}.

Here we pursue some natural next steps in this program, focusing on $\lstw$ and $\s_{16}$ in motivic weights 13 and 15, respectively. Throughout, we consider rational cohomology groups of algebraic varieties and Deligne--Mumford stacks over $\Q$ with their associated motivic structures, i.e., a rational mixed Hodge structure and an action of $\Gal(\overline \Q | \Q)$ on the extension of scalars to $\Q_\ell$.  The structures $\l := H^2(\mathbb{P}^1)$, $\s_{k+1} := W_k H^k(\M_{1,k})$, and $\lstw := \l \otimes_\Q \s_{12}$ play a special role. We say that a cohomology group is of \emph{Tate type} if the associated graded of its weight filtration is a direct sum of tensor powers of $\l$. If $X$ is smooth then, by Poincar\'e duality, $H^*(X)$ is of Tate type if and only if $H^*_c(X)$ is of Tate type.

\subsection{Weight 13 cohomology} 

In \cite[Section~4]{CLPW}, we gave a  presentation for $H^{13}(\Mb_{g,n})$ for all $g$ and $n$. Surprisingly, this presentation stabilizes with the genus; it is independent of $g$, for $g \geq 2$.  Here, we use this presentation to produce a graph complex that computes $\gr^W_{13}H^*_c(\M_{g,n})$ and study its cohomology for small $g$ and $n$. 

Recall that $H^*(\M_{g,n})$ is of Tate type for $3g +2n \leq 24$ \cite[Theorem~1.5]{CLPW}. Thus $\gr_k^W H^*(\M_{g,n})$ vanishes when $k$ is odd and $3g + 2n \leq 24$. We improve this vanishing bound by one for $k = 13$.

\begin{prop}\label{prop:wt 13 vanishing}
    If $3g + 2n \leq 25$ then $\gr_{13}^W H^*_c(\M_{g,n}) = 0$.
\end{prop}

We also compute $\gr^W_{13}H^*_c(\M_{g,n})$ in the first cases where it does not vanish, when $3g + 2n$ is $26$ or $27$. Let $\delta_{0,n}$ denote the Kronecker delta function, i.e., $\delta_{0,n} = 1$ if $n = 0$ and is $0$ otherwise. For a partition $\lambda$ of a positive integer $n$, let $V_\lambda$ denote the corresponding Specht module, which is an irreducible representation of $\ss_n$.

\begin{thm} \label{thm:lowexc13}
    Suppose $3g + 2n \in \{26, 27\}$. Then $\gr_{13}^W H^*_c(\M_{g,n})$ is nonzero only in degree $$k(g,n) = 3g + n - 2 - \delta_{0,n},$$ and there is an $\ss_n$-equivariant isomorphism $\gr_{13}^W H^{k(g,n)}_c(\cM_{g,n}) \cong Z_{g,n} \otimes \lstw$, where
\begin{align*}
    Z_{1,12} & \cong V_{21^{10}} & Z_{2,10} & \cong V_{1^{10}} & Z_{3,9} & \cong V_{1^{9}} \\
    Z_{4,7} & \cong V_{1^{7}} & Z_{5,6} & \cong V_{1^{6}} \oplus V_{21^4}^{\oplus 2} & Z_{6,4} & \cong V_{1^{4}} \\ Z_{7,3} & \cong V_{1^{3}} \oplus V_{21}^{\oplus 2}& Z_{8,1} & \cong \Q & Z_{9,0} & \cong \Q 
\end{align*}
In particular, we have $\gr_{13}^W H^{24}_c(\M_9) \cong \lstw$.
\end{thm}

\noindent  Note that  $\gr_2^W H^{24}_c(\M_9)$ is also nonzero; the shift of $W_0H^6_c(\M_3) \wedge W_0H^{15}_c(\M_6)$ appearing in the first term of  \cite[Theorem~1.2]{PayneWillwacher21} contributes a summand isomorphic to $\l$.

\subsection{Weight 15 cohomology and appearances of \texorpdfstring{$\s_{16}$}{S16}}

Next, we describe the $\s_{16}$-isotypic part of $H^{15}(\Mb_{g,n})$. Note that $\s_{16}$ is not a subquotient of the cohomology of $\Mb_{g,n}$ for $g \neq 1$, by \cite[Theorem~1.1]{CLP-STE}. Hence, the essential case is when $g = 1$. There is a canonical splitting
\[
H^{15}(\Mb_{1,15}) = \mathsf{S}_{16} \oplus V
\]
where $V^\semis \cong \bigoplus \mathsf{L}^2\mathsf{S}_{12},$ by \cite[Lemma 8.3]{CLP-STE}. The subspace $\mathsf{S}_{16} \subset H^{15}(\Mb_{1,15})$ is the image of 
\[
H^{15}_c(\M_{1,15}) \rightarrow  H^{15}(\Mb_{1,15}),
\]
or, equivalently, the kernel of the restriction map $H^{15}(\Mb_{1,15}) \to H^{15}(\partial \M_{1,15})$. The following description of $H^{15}(\Mb_{1,n})$ is well-known to experts (cf. \cite{Getzler, Petersenappendix}).

Let $K^{15}_n := V_{(n-14)1^{14}}$ denote the $\ss_n$-representation associated to the hook shape Young diagram of size $n$ with 15 boxes in the first column. (For $n < 15$, we set $K^{15}_n := 0.)$
It is an irreducible $\ss_n$-representation of dimension $\binom{n-1}{14}$ generated by elements $k_P$, for ordered subsets $P \subset \{1, \ldots, n\}$ of size $15$, with relations  
\begin{equation} \label{11rels} 
k_{\sigma(P)} = \sgn(\sigma) \cdot k_P, \quad \mbox{and} \quad \sum_{j = 0}^{15}(-1)^j \cdot k_{\{a_0, \ldots, \widehat{a_j}, \ldots, a_{15}\}} = 0,
\end{equation}
for any permutation $\sigma$ of $P$, and any size $16$ ordered subset $\{a_0, \ldots, a_{15} \} \subset \{1, \ldots, n \}$.

\begin{thm} \label{thm:H15}
There is an $\ss_n$-equivariant direct sum decomposition of the form
\[
H^{15}(\Mb_{1,n}) \cong (K^{15}_n \otimes_\Q \mathsf{S}_{16}) \oplus V_n,
\]
where the underlying motivic structure of the second summand is $V_n^\semis \cong \bigoplus \mathsf{L}^2\mathsf{S}_{12}$.
\end{thm}

\noindent The isomorphism identifies $(\Q \cdot k_P) \otimes \s_{16}$ with the pullback of $\s_{16} \subset H^{15}(\Mb_{1,P})$ under the forgetful map $\Mb_{1,n} \to \Mb_{1,P}$.

We now focus on the $\s_{16}$-isotypic subspace of $\gr_{15}^W H^*_c(\M_{g,n})$. By Theorem~\ref{thm:H15}, we have a direct sum decomposition of Hodge structures or $\ell$-adic Galois representations
\begin{equation} \label{eq:U+W}
\gr^W_{15} H^*_c(\M_{g,n}) \cong U_{g,n} \oplus W_{g,n}, 
\end{equation}
where
\[
U_{g,n}^{\semis} \cong \bigoplus \s_{16} \quad \mbox{and} \quad W_{g,n}^{\semis} \cong \bigoplus \l^2 \s_{12}.
\]
Recall also that $\gr_{15}^W H^*_c(\M_{g,n}) = 0$ for $3g + 2n \leq 24$, by \cite[Theorem~1.5]{CLPW}. 
We have the following improved vanishing result for $U_{g,n}$.

\begin{prop}\label{prop:wt 15 vanishing}
    If $3g + 2n \leq 32$ then $U_{g,n} = 0$.
\end{prop}

\noindent In particular, in the key special case where there are no marked points, the motivic structure $\s_{16}$ does not appear in $\gr_{15}^WH^*_c(\M_{g})$ for $g \leq 10$.  

Following the usual convention, we omit the second subscript when $n = 0$ and write $U_g := U_{g,0}$. Let $U^k_{g} \subset \gr^W_{15} H^k_c(\M_{g})$ denote the degree $k$ part of $U_{g}$. 

\begin{thm} \label{thm:s16}
    For $g = 11$, the $\s_{16}$-isotypic subspace of $\gr^W_{15}H^*_c(\M_{11})$ is given by
    \[
    U_{11}^k \cong \begin{cases}
        \s_{16} & \mbox{ for } k = 30, \\
        0 & \mbox{ otherwise.}
    \end{cases}
    \]
    For $g = 12$, we have $U_{12} = 0$.
\end{thm}

\noindent Thus, we have shown that the middle cohomology group $H^{30}(\M_{11})$ does not vanish, which was not previously known. See \S\ref{sec:low excess15} for a complete computation of $U_{g,n}$ in the cases where $3g + 2n \leq 36$.

Although we do not yet have a complete description of $U_{g}$, e.g., in terms of shifts of symmetric and exterior powers of weight zero cohomology groups of smaller moduli spaces, as in \cite[Theorem~1.2]{PayneWillwacher21}, we produce a large subspace of this form.

Let $\V_g^{r,k}$ denote the degree $k$ and genus $g$ part of the $r$-fold symmetric power
      \[
    \Sym^r\bigg(\bigoplus_{h} W_0H_c^*(\M_h)\bigg).  
    \]
Here we follow the Koszul sign convention, so the underlying graded vector space of the symmetric algebra on a graded vector space $U^*$ is $\Sym^*(U^{\mathrm{even}}) \otimes \bigwedge^*(U^{\mathrm{odd}})$.

\begin{thm}\label{thm:emb intro}
    There is an injective map
\begin{equation} \label{eq:inj}
\resizebox{.92\hsize}{!}{
$\begin{aligned}
     \big( \V^{14,k-29}_{g-1} 
    \oplus
    \V^{14,k-30}_{g-2} 
    \oplus
    \V^{13,k-30}_{g-3} 
    \oplus
    \V^{10,k-30}_{g-5} 
    \oplus
    \V^{7,k-30}_{g-7} 
    \oplus
    \V^{4,k-30}_{g-9} 
    \oplus
    \V^{1,k-30}_{g-11} 
     \big) \otimes \s_{16} 
     \to U^k_g.
    \end{aligned}$
    }
\end{equation}
  \end{thm}

  \noindent The injection we construct comes from a subcomplex of a graph complex that computes $U_g$; see \S\ref{sec:inj} for details. The map \eqref{eq:inj} is not surjective; for example, $U_{11}$ is not in the image.

Theorem~\ref{thm:emb intro} yields many new nonvanishing results and lower bounds on the dimensions of cohomology groups of $\M_g$.  For instance, the exponential growth of $\dim H^{2g+20}_c(\M_g)$, which was not previously known, is deduced from the summand $\V^{10,k-30}_{g-5}$ using the exponential growth of $\dim W_0 H^{2g}_c(\M_g)$. We similarly prove exponential growth of $\dim H^{2g+k}_c(\M_g)$ for 20 other values of $k$ for which this was not previously known; see Corollary~\ref{cor:exp growth}. 

The following picture illustrates the values of $k$ for which $\dim_\Q H^{2g+k}_c(\M_g)$ is now known to grow at least exponentially with $k$, with dark gray boxes for the previously known cases (from weight 0, 2 and 11) and light gray boxes for the new contributions from $U_g$.

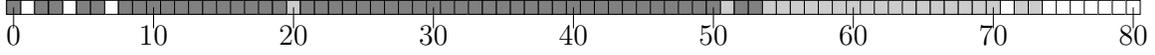
\[
\resizebox{.95\hsize}{!}{
\begin{tikzpicture}
  \begin{scope}[yshift=.2cm,scale=.2]
    \foreach \x in {0, 1, 2, 3, 4, 5, 6, 7, 9, 10, 13, 20, 51, 54, 55, 56, 57, 58, 59, 60, 61, 62, 63, 64, 65, 66, 67, 68, 69, 70,71,72,73,74,75,76,77,78,79,80}
    \node[minimum size=2mm, draw=black, fill=white, inner sep=0] at (\x,0) {};
  \end{scope}
    \begin{scope}[yshift=.2cm, scale=.2]
        \foreach \x in {8, 11, 12, 15, 16, 18, 19, 20, 21, 22, 23, 24, 25, 26, 27, 28, 29, 30, 31, 32, 33, 34, 35, 36, 37, 38, 39, 40, 41, 42, 43, 44, 45, 46, 47, 48, 49, 50, 51, 52, 53, 54, 55, 56, 57, 58, 59, 60, 61, 62, 63, 64, 65, 66, 67, 68, 69, 70, 72, 73}
        \node[minimum size=2mm, draw=black, fill=black!20, inner sep=0] at (\x,0) {};
        \end{scope}
  \begin{scope}[yshift=.2cm, scale=.2]
    \foreach \x in {8, 11, 12, 14, 15, 16, 17, 18, 19, 21, 22, 23, 24, 25, 26, 27, 28, 29, 30, 31, 32, 33, 34, 35, 36, 37, 38, 39, 40, 41, 42, 43, 44, 45, 46, 47, 48, 49, 50, 52, 53}
    \node[minimum size=2mm, draw=black, fill=black!50, inner sep=0] at (\x,0) {};
    \end{scope}
    \begin{scope}[yshift=.2cm, scale=.2]
      \foreach \x in {2, 3, 5, 6, 8, 9, 10, 12, 13}
      \node[minimum size=2mm, draw=black, fill=black!50, inner sep=0] at (\x,0) {};
      \end{scope}
  \begin{scope}[yshift=.2cm, scale=.2]
    \foreach \x in {0,3}
    \node[minimum size=2mm, draw=black, fill=black!50, inner sep=0] at (\x,0) {};
    \end{scope}
\begin{scope}[scale=.2]
\foreach \x in {0,10,20,30,40,50,60, 70, 80}
\draw (\x,1) -- (\x,-.5) (\x,-1) node {\x};
\end{scope}
\end{tikzpicture}
}
\]

\noindent This provides further evidence for the conjecture that $\dim H^{2g+k}_c(\M_g)$ grows at least exponentially with $g$ for all but finitely many non-negative integers $k$ \cite[Conjecture~1.3]{PayneWillwacher24}. The analogous statement for moduli spaces of principally polarized abelian varieties is a recent theorem: the dimension of $H^{2g+k}_c(\mathcal{A}_g)$ grows at least exponentially with $g$ for all but finitely many non-negative integers $k$ \cite[Corollary~1.5]{BCGP}.

Our study of the multiplicity of $\s_{16}$ in $\gr_{15}^W H^*_c(\M_{g,n})$ also leads naturally to a generating function for the corresponding Euler characteristic. See Theorem~\ref{thm:chi-S16}. The result is similar to the generating function for the weight $11$ Euler characteristic of $\M_{g,n}$ from \cite[Section~7]{PayneWillwacher24}.

\section{Graph complexes in weight 13} \label{sec:weight13}
In this section, we define and study graph complexes that compute the weight $13$ graded piece of $H^*_c(\M_{g,n})$. We start with the Getzler--Kapronov graph complex, obtained as the Feynman transform of the modular operad $H^{13}(\MM)$, and then pass to a simpler quasi-isomorphic graph complex whose generators we explicitly describe.
\subsection{Reminder of the Feynman transform}
\label{sec:feyn recollection}
Recall, e.g. from \cite[Section~2]{PayneWillwacher21}, that the associated graded of the weight filtration 
\[\gr^W_k H^*_c(\M_{g,n}) := W_k H^*_c(\M_{g,n})/W_{k-1}H^*_c(\M_{g,n})\]
is identified with the cohomology of the Feynman transform of $H(\MM)$, the modular cooperad whose $(g,n)$ part is $H^*(\MM_{g,n})$. 
This Feynman transform is also known as the Getzler--Kapranov graph complex and is denoted
\begin{equation}\label{equ:GKdef}
\GK_{g,n} 
\bigg( \bigoplus_{
\Gamma\text{ stable graph}} 
\Big( \bigotimes_{v\in V(\Gamma)}H^{*}(\MM_{g_v,n_v})\otimes \Q[-1]^{\otimes |E(\Gamma)|}\Big) \bigg) \Big/ \sim ,
\end{equation}
where the sum is over stable graphs of type $(g,n)$.  The relations in \eqref{equ:GKdef} are induced by isomorphisms of stable graphs; the end result is isomorphic to a direct sum over representatives of the finitely many isomorphism classes of stable graphs of type $(g,n)$, and the contribution of $[\Gamma]$ is the $\Aut(\Gamma)$-coinvariants of the parenthesized expression. The  differential is induced by pullback along the tautological clutching and gluing maps.

Each element of $\GK_{g,n}$ is represented by a linear combination of stable graphs with $n$ numbered legs, such that the sum of the genera $g_v$ of the vertices plus the first Betti number of the graph is $g$, and such that each vertex $v$ of valence $n_v$ is decorated by an element of $H^{*}(\MM_{g_v,n_v})$.  When each vertex is decorated by a homogeneous element of cohomological degree $k_v$, we say that the \emph{weight} of the decorated graph is $k = \sum_v k_v$.

Let $\GK_{g,n}^k$ denote the subspace spanned by generators of weight $k$.  The differential preserves these subspaces, giving a direct sum decomposition $\GK_{g,n} \cong \bigoplus_k \GK_{g,n}^k$.  By construction $\GK_{g,n}^k$ is isomorphic to the weight $k$ row of the $E_1$-page of the weight spectral sequence associated to the inclusion $\M_{g,n} \subset \MM_{g,n}$. 
  
The differential on $\GK^k_{g,n}$ has two terms,
\[
d = d_\vartheta + d_\xi.
\]
The piece $d_\xi$ acts by creating a self-loop at a vertex
\[
d_\xi \colon 
\begin{tikzpicture}
    \node[ext, label=90:{$x$}] (v) at (0,0) {$\scriptstyle g_v$};
    \draw (v) edge +(-.5,-.5) edge +(.5,-.5) edge +(0,-.5)
    edge +(-.25,-.5) edge +(.25,-.5) ;
\end{tikzpicture}
\to 
\begin{tikzpicture}
    \node[ext, label=90:{$\scriptstyle \xi^* x$}] (v) at (0,0) {$\scriptscriptstyle g_v-1$};
    \draw (v) edge +(-.5,-.5) edge +(.5,-.5) edge +(0,-.5)
    edge +(-.25,-.5) edge +(.25,-.5) edge[loop] (v);
\end{tikzpicture}.
\]
The decoration at the new vertex, which has a genus $g_v-1$ and valence $n_v+2$, is obtained from the decoration $x\in H^{k_v}(\MM_{g_v,n_v})$ via pullback along the morphism $\xi\colon \MM_{g_v-1,n_v+2} \to \MM_{g_v,n_v}$.

Similarly, the piece $d_\vartheta$ acts as the sum over all vertices of all possible ways of splitting a vertex into two vertices joined by an edge, 
\[
d_\vartheta \colon 
\begin{tikzpicture}
    \node[ext, label=90:{$x$}] (v) at (0,0) {$\scriptstyle g_v$};
    \draw (v) edge +(-.5,-.5) edge +(.5,-.5) edge +(0,-.5)
    edge +(-.25,-.5) edge +(.25,-.5) ;
\end{tikzpicture}
\to 
\sum
\begin{tikzpicture}
    \node[ext, label=90:{$x'$}] (v) at (0,0) {$\scriptstyle h'$};
    \node[ext, label=90:{$x''$}] (w) at (.7,0) {$\scriptstyle h''$};
    \draw (v) edge +(-.5,-.5) edge +(-.25,-.5) edge (w)
    (w) edge +(.5,-.5) edge +(0,-.5)
     edge +(.25,-.5) ;
\end{tikzpicture}
\]
so that the new decoration is the pullback 
\[
\sum x' \otimes x'' = \vartheta^* x
\]
of $x\in H^{k_v}(\MM_{g_v,n_v})$ under the appropriate gluing morphism 
$\vartheta :\MM_{h',n'}\times \MM_{h'',n''} \to\MM_{g_v,n_v}$.
Here $h'+h''=g_v$ and $n'+n''=n_v+2$. 
We are interested in the case of weight $k=13$. Since $H^{k'}(\MM_{g,n})=0$ for odd $k'<11$ \cite{BergstromFaberPayne}, generators of $\GK_{g,n}^{13}$ are of the following two forms:

\begin{enumerate}
    \item[A.] Graphs with one vertex $v$ of weight $k_v=13$;
    \item[B.] Graphs with one vertex $v$ of weight $k_v=11$ and one vertex $u$ of weight $k_u=2$.
\end{enumerate}
   In either case, all other vertices have weight zero.  
   We also note that a weight 11 vertex must have genus $g_v=1$ and valence $n_v\geq 11$, by  \cite[Theorem 1.1]{CanningLarsonPayne}.
   Similarly, a weight $13$ vertex $w$ must have genus $g_w \geq 1$ and valence $n_w \geq 10$ by \cite[Theorem~1.7]{CLPW}.

\subsection{Simplifying \texorpdfstring{$\GK^{13}_{g,n}$}{GK13gn}}
Next, we simplify the graph complex $\GK^{13}_{g,n}$, following the analogous construction in weight 2 \cite{PayneWillwacher21}.
We define the graded subspace 
\[
\tilde I_{g,n} \subset \GK^{13}_{g,n}
\]
to be the subspace spanned by decorated graphs with one or more of the following features:
\begin{enumerate}
\item A weight $0$ vertex of positive genus.
\item A weight $2$ vertex $2$ of genus $g_u\geq 2$.
\item A weight $2$ vertex of genus $g_v= 1$ and valence $n_v\geq 2$, or of genus $g_v=1$ and valence $n_v=1$ such that the unique neighbor is a vertex of weight 0.
\item A weight 13 vertex $w$ of genus $g_w\geq 2$.
\end{enumerate}
We also recall that a weight 11 vertex must necessarily have genus $1$ for the graph to be a non-zero element of $\GK_{g,n}^{13}$.
We then define 
\[
I_{g,n} := \tilde I_{g,n} + d\tilde I_{g,n}\subset \GK^{13}_{g,n}
\]
to be the differential closure of $\tilde I_{g,n}$, and we define the quotient complex 
\[
\bGK^{13}_{g,n}:= \GK^{13}_{g,n} / I_{g,n}.
\]
The goal is to show the following proposition, to be proven in Section \ref{sec:prop bGK proof} below.

\begin{prop}\label{prop:bGK}
    The projection $
    \GK^{13}_{g,n} \to \bGK^{13}_{g,n}
    $ is a quasi-isomorphism.
\end{prop}

For now, we give a more explicit description of $I_{g,n}$.
\begin{lem}\label{lem:Ign gen}
The subspace $I_{g,n}\subset \GK^{13}_{g,n}$ is spanned by decorated graphs possessing at least one of the features (1)--(4) above, together with the following graphs:
\begin{enumerate}
\setcounter{enumi}{4}
\item Graphs with a loop attached to a weight 0 vertex.
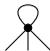
\[
\begin{tikzpicture}
\node[int] (v) at (0,0) {};
\draw (v) edge[loop] (v) edge +(-.3,-.3) edge +(0,-.3) edge +(.3,-.3);
\end{tikzpicture}
\]
\item Graphs with a loop attached to a weight 2 vertex $u$ of genus $g_u=0$ and decoration in the image of the pullback under the corresponding gluing map  $\xi\colon \MM_{0,n} \to \MM_{1,n_u-2}$. 
\item Linear combinations $(d_\xi+d_\vartheta)\Gamma$, where $\Gamma$ has a weight 13 vertex of genus $2$. 
\end{enumerate}
\end{lem}
\begin{proof}
    The subspace $I_{g,n}$ is evidently generated by graphs having one the features (1)--(4) above, and elements $d\Gamma$, where $\Gamma$ is again such a graph.
    
    First suppose that $\Gamma$ has feature (1), i.e., a weight zero vertex $w$ of genus $g_w\geq 1$.
   If $g_w > 1$, then all terms in $d\Gamma$ also have this feature. If $g_w = 1$, then all terms in $d\Gamma$ have this feature, except for one arising from $d_\xi \Gamma$, that corresponds to creating a self-loop at $w$ and creates feature (5). Thus $d\Gamma$ lies in the span of graphs with features (1) and (5). Conversely, given any graph $\Gamma'$ with feature (5), let $\Gamma$ be
   be the graph that contracts the self-loop and replaces it with a genus $1$ vertex shows. Then $\Gamma'$ is a term in $d\Gamma$ and all other terms have already been shown to lie in $I_{g,n}$, so $\Gamma'$ is also in $I_{g,n}$.
    
    Next suppose $\Gamma$ has feature (2), i.e., a weight 2 vertex $u$ of genus $g_u\geq 2$. Then any term in $d\Gamma$ has a weight 2 vertex of genus $\geq 2$, a weight zero vertex of genus $\geq 1$, or a weight 2 vertex of genus 1 with a self-loop, and hence necessarily of valence $\geq 3$. (Note that one can never have a $(2,0)$ vertex because $H^{13}(\Mb_{2,0}) = 0$.) Hence $d\Gamma\in \tilde I_{g,n}$.

    Next suppose $\Gamma$ has feature (3), i.e., a weight 2 vertex $u$ of genus 1 as stated.
    Then $d_\vartheta\Gamma$ is a linear combination of terms that again have feature (2) or (1), so that $d_\vartheta \Gamma\in \tilde I_{g,n}$. Likewise, the only term in $d_\xi\Gamma$ that is not readily in $\tilde I_{g,n}$ is obtained by creating the self-loop at $u$, giving rise to the graphs having feature (6). Thus, $d\Gamma$ lies in the span of graphs with features $(1), (2)$, and $(6)$. Conversely, given a graph $\Gamma'$ with feature (6), let $\Gamma$ be the graph that contracts the self-loop at $u$ and decorates the resulting genus $1$ vertex with the class guaranteed to pullback to the decoration at $u$. Because $H^2(\Mb_{0,3}) = 0$, the vertex $u$ must have valence at least $4$; this ensures that the graph $\Gamma$ has feature (3). By construction, $\Gamma'$ is a term in $d\Gamma$ and all other terms have already been shown to lie in $I_{g,n}$, so $\Gamma'$ lies in $I_{g,n}$.
    
    Finally, if $\Gamma$ has feature (4) and $g_v \geq 3$, then $d \Gamma$ is again a linear combination of graphs having features (2) and (4). If $g_v = 2$, then $d \Gamma$ is of type (7).
\end{proof}

\subsection{The part of type (12,1)} \label{pics}
Each term in the Getzler--Kapranov graph complex comes with an associated $\qq$-Hodge structure, and the differentials respect the Hodge decomposition after tensoring with $\CC$. Hence, we have a Hodge decomposition 
\[
\mathsf{GK}_{g,n}^{k}\otimes \cc = \bigoplus_{p+q=k} \mathsf{GK}_{g,n}^{p,q}.
\]
Specializing to $k=13$, we have
\[
\GK^{13}_{g,n}\otimes \cc = \GK^{12,1}_{g,n}\oplus \GK^{1,12}_{g,n},
\]
because the rank 2 Hodge structure $\lstw$ has  types $(12,1)$ and $(1,12)$. 
Similarly,
\begin{align*}
I_{g,n}\otimes \cc &= I_{g,n}^{12,1}\oplus I_{g,n}^{1,12}
&
\bGK^{13}_{g,n}\otimes \cc &= \bGK^{12,1}_{g,n}\oplus \bGK^{1,12}_{g,n},
\end{align*}
with $\bGK^{12,1}_{g,n}=\GK^{12,1}_{g,n}/I_{g,n}^{12,1}$ and $\bGK_{g,n}^{1,12}=\GK_{g,n}^{1,12}/I_{g,n}^{1,12}$.

We introduce graphical notation for generators of $\bGK^{12,1}_{g,n}$, following \cite{PayneWillwacher21, PayneWillwacher24}.
Generators are decorated graphs, and graphs having features (1)--(7) are set to zero. Because of (1), we may assume that all weight 0 vertices are decorated by $1\in H^0(\MM_{0,n})$, omit the decoration and genus, and just draw a black vertex.
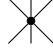
\[
\begin{tikzpicture}
    \node[int] (v) at (0,0) {};
    \draw (v) edge +(-.3,-.3) edge +(0,-.3) edge +(.3,-.3)              edge +(-.3,.3) edge +(0,.3) edge +(.3,.3);
\end{tikzpicture}
\]
Similarly, because of (5), we may assume that there are no loops attached to weight zero vertices. 
Because of (2), we may assume that all weight $2$ vertices have genus $\leq 1$.
The decorations $\delta_S$ on a weight two vertex of genus zero can be graphically represented by drawing the vertex as two vertices, connected by a special edge 
\begin{equation}
\label{equ:delta crossed}
\begin{tikzpicture}
    \node[int, label=90:{$\delta_S$}] (v) at (0,0) {};
    \draw (v) edge +(-.3,-.3) edge +(-.3,.3) edge +(-.3,0)              edge +(.3,-.3) edge +(.3,0) edge +(.3,.3);
\end{tikzpicture}
=:
\begin{tikzpicture}
    \node[int] (v) at (0,0) {};
    \node[int] (w) at (0.5,0) {};
    \draw (v) edge +(-.3,-.3) edge +(-.3,.3) edge +(-.3,0) edge[crossed] (w) 
    (w) edge +(.3,-.3) edge +(.3,0) edge +(.3,.3);
\end{tikzpicture}
\end{equation}

Because of (3), we may assume all weight $2$, genus $1$ vertices have valence $1$ and represent them by  a crossed self-edge 
\[
\begin{tikzpicture}
    \node[ext, label=90:{$\delta_{irr}$}] (v) at (0,0) {};
    \draw (v) edge +(0,-.3);
\end{tikzpicture}
=:
\begin{tikzpicture}
    \node[int] (v) at (0,0) {};
    \draw (v) edge +(0,-.3) edge[loop, crossed] (v);
\end{tikzpicture}
\]
We draw the generator $\omega_A$ at the weight $11$ vertex (of genus 1) as
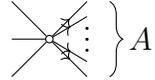
\[
\begin{tikzpicture}
    \node[ext] (v) at (0,0) {};
    \node at (.5, .1) {$\scriptscriptstyle \vdots$};
    \draw (v) edge +(-.5,-.5) edge +(-.5,0) edge +(-.5,+.5) 
    edge[->-] +(.5,+.5) edge[->-] +(.5,+.3) edge[->-] +(.5,-.3) edge[->-] +(.5,-.5)
    ;
    \draw [decorate,decoration={brace,amplitude=5pt}]
  (.8,.5) -- (.8,-.5) node[midway,xshift=1em]{$A$};
\end{tikzpicture}
\]

For the weight $13$ vertex of genus $1$, we recall the presentation of $H^{12,1}(\Mb_{1,n})$ in \cite{CLPW}. 
Consider the parametrized boundary divisors of $\Mb_{1,n}$ of the form
\[\iota_{A} \colon \Mb_{1, A \cup p} \times \Mb_{0,A^c \cup p'} \to \Mb_{1,n}.\]
For each such boundary divisor, there is a push forward map
\[
\iota_{A*} \colon H^{11,0}(\Mb_{1,A\cup p}) \otimes H^0(\Mb_{g-1,A^c \cup p'}) \subset H^{11,0}(\Mb_{1, A \cup p} \times \Mb_{0,A^c \cup p'}) \to H^{12,1}(\Mb_{1,n}). 
\]
We consider classes of the form
\begin{equation} \label{zdef} 
Z_{B \subset A} := \iota_{A*}(\omega_{B \cup p}\otimes 1),
\end{equation}
where $B \subset A \subset \{1, \ldots, n\}$ with $B$ increasing of size $10$. Here, $\omega_{B\cup p}$ is the pullback under the forgetful map $\Mb_{1,n}\rightarrow \Mb_{1,B\cup p}$ of the generator $\omega\in H^{11,0}(\Mb_{1,11})$ corresponding to the Ramanujan cusp form. The elements $Z_{B\subset A}$ generate $H^{12,1}(\Mb_{1,n})$, by \cite[Lemma 4.11]{CLPW}. We depict the weight 13 vertex (of genus 1) decorated by the generator $Z_{B\subset A}\in H^{12,1}(\MM_{1,n})$ by
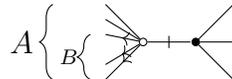
\[
\begin{tikzpicture}
    \node[ext] (v) at (0,0) {};
    \node[int] (w) at (0.7,0) {};
    \draw (v) edge +(-.5,.5) edge +(-.5,.3) edge[->-] +(-.5,.1) edge[crossed] (w) 
    (w) edge +(.5,-.5) edge +(.5,0) edge +(.5,.5);
    \draw (v) edge[->-] +(-.5,-.3) edge[->-] +(-.5,-.5);
    \draw [decorate,decoration={brace,amplitude=5pt}]
  (-1.2,-.5) -- (-1.2,.5) node[midway,xshift=-1em]{$A$};
    \draw [decorate,decoration={brace,amplitude=5pt}]
  (-.7,-.5) -- (-.7,.1) node[midway,xshift=-.7em]{$\scriptstyle B$};
\end{tikzpicture}
\]

We use these drawings to make the relations (7) more precise.

\begin{lem}\label{lem:Ign 2} 
The graded subspace $I_{g,n}^{12,1}$ is spanned by graphs $\Gamma\in \GK_{g,n}^{12,1}$ that satisfy either of the conditions (1)--(6) above, and in addition by the following elements:
\begin{enumerate}
    \item[($7'$)] Graphs with a weight 13 vertex of genus 1 with a self-edge $(s,t)$, decorated by $Z_{B\subset A}$ such that $|A^c|\geq 3$ and $s,t\in A^c$.
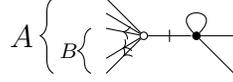
\[
\begin{tikzpicture}
    \node[ext] (v) at (0,0) {};
    \node[int] (w) at (0.7,0) {};
    \draw (v) edge +(-.5,.5) edge +(-.5,.3) edge[->-] +(-.5,.1) edge[crossed] (w) 
    (w) edge +(.5,-.5) edge +(.5,0) edge[loop] (w);
    \draw (v) edge[->-] +(-.5,-.3) edge[->-] +(-.5,-.5);
    \draw [decorate,decoration={brace,amplitude=5pt}]
  (-1.2,-.5) -- (-1.2,.5) node[midway,xshift=-1em]{$A$};
    \draw [decorate,decoration={brace,amplitude=5pt}]
  (-.7,-.5) -- (-.7,.1) node[midway,xshift=-.7em]{$\scriptstyle B$};
\end{tikzpicture}
\]
    \item[($7''$)] Linear combinations $\Gamma-\frac{1}{12}\Gamma'$ where $\Gamma$ is a graph as in (7$'$) but with $A^c=\{s,t\}$, and $\Gamma'$ is the same graph as $\Gamma$ except that the weight 13 vertex is replaced by a weight 11 vertex and a $\delta_{irr}$-decorated weight 2 vertex:
    
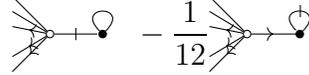
\[
\begin{tikzpicture}
    \node[ext] (v) at (0,0) {};
    \node[int] (w) at (0.7,0) {};
    \draw (v) edge +(-.5,.5) edge +(-.5,.3) edge[->-] +(-.5,.1) edge[crossed] (w) 
    (w)  edge[loop] (w);
    \draw (v) edge[->-] +(-.5,-.3) edge[->-] +(-.5,-.5);
\end{tikzpicture}
-
\frac1{12}
\begin{tikzpicture}
    \node[ext] (v) at (0,0) {};
    \node[int] (w) at (0.7,0) {};
    \draw (v) edge +(-.5,.5) edge +(-.5,.3) edge[->-] +(-.5,.1) edge[->-] (w) 
    (w) edge[loop, crossed] (w);
    \draw (v) edge[->-] +(-.5,-.3) edge[->-] +(-.5,-.5);
\end{tikzpicture}
    \]
\end{enumerate}
\end{lem}
Note that the relations (7$'$) and (7$''$) allow us to remove all generators from $\bGK_{g,n}^{13}$ that have a $Z_{B\subset A}$-decorated weight 13 vertex with a self edge $(s,t)$ such that $s,t\in A^c$. We just need to keep in mind that if $A^c=\{s,t\}$ these graphs are not zero in $\bGK_{g,n}^{13}$, rather they are identified with the right hand term in $(7'')$.

\begin{proof}
    The argument is similar to the proof of Lemma \ref{lem:Ign gen}, but we must revisit the generators of the form (7). We consider expressions $d\Gamma$ where $\Gamma$ has a weight 13 vertex $v$ of genus 2 decorated with $Z_{B \subset A}$.
    The only terms in $d\Gamma$ that do not have any of the forms (1)--(6) above are as follows:
    First, $d_\xi\Gamma$ has a term creating a self-edge at $v$, that is of the form  
    
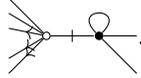
\[
\begin{tikzpicture}
    \node[ext] (v) at (0,0) {};
    \node[int] (w) at (0.7,0) {};
    \draw (v) edge +(-.5,.5) edge +(-.5,.3) edge[->-] +(-.5,.1) edge[crossed] (w) 
    (w) edge +(.5,-.5) edge +(.5,0) edge[loop] (w);
    \draw (v) edge[->-] +(-.5,-.3) edge[->-] +(-.5,-.5);
\end{tikzpicture}.
\]
Second, $d_\vartheta\Gamma$ has terms from splitting $v$ into a weight 11 and a weight 2 vertex, each of genus 1. 
Because of the genus 1 vertex of weight 2, these terms are of the form (3) above, except in the special case that the new weight 2 vertex has valence 1.
When the new weight $2$ vertex has valence $1$, the contribution to $d_\vartheta \Gamma$ is determined by setting $S = \varnothing$ in \cite[Lemma~4.5]{CLPW}. We now consider two cases.

\textit{Case 1:} Suppose our starting decoration has $A^c \neq \varnothing$. Then the term in $d_\vartheta$ where the new weight $2$ vertex has valence $1$
is zero by \cite[Lemma~4.5]{CLPW}. In this case,
the only term in $d\Gamma$ not already known to be in $I_{g,n}$ is $d_\xi \Gamma$, which has the form (7$'$). In this case, the vertex on the right has valence at least $3$ because the $A^c$ for our original decoration was non-empty. Conversely, given a graph $\Gamma'$ as in (7$'$), contracting the self-loop and replacing it with a genus $1$ vertex gives a graph $\Gamma$ such that $d\Gamma$ is $\Gamma'$ plus terms already known to be in $I_{g,n}^{12,1}$.

\textit{Case 2:} Suppose our starting decoration has $A^c = \varnothing$. Then by \cite[Lemma~4.5]{CLPW}, the term in $d_{\vartheta}$ in which the new weight $2$ vertex has valence $1$ is decorated by $-\psi\in H^2(\Mb_{1,1})$. But in $H^2(\MM_{1,1})$, we have $\psi = \frac 1 {12}\delta_{irr}$, by \cite[Theorem 2.2]{ArbarelloCornalba}. This yields the second graph pictured in (7$''$).
Since $A^c = \varnothing$, the first graph pictured in (7$''$) is
$d_\xi \Gamma$.
\end{proof}

\subsection{One-vertex complex}
The proof of Proposition \ref{prop:bGK} will use the spectral sequence on $\GK^{13}_{g,n}$ from the filtration by the number of vertices.
To this end, we need to study the one-vertex complexes 
\[
V_{g,n}^k = \left( 
H^k(\MM_{g,n}) \xrightarrow{\xi^*}
H^k(\MM_{g-1,n+2})_{\ss_2} \xrightarrow{\xi^*} \cdots 
\xrightarrow{\xi^*}
H^k(\MM_{g-h,n+2h})_{\ss_2\wr \ss_h}
\xrightarrow{\xi^*} \cdots
\right)
\]
introduced in \cite[Section 4]{PayneWillwacher21}. 
As in loc. cit., the action of the wreath product $\mathbb{S}_2 \wr \mathbb{S}_h=\mathbb S_h \ltimes \mathbb S_2^h$ is understood to include a sign that accounts for the permutation of the tensor factors $\mathbb{Q}[-1]^{\otimes h}$ in \eqref{equ:GKdef}. 
From \cite{PayneWillwacher21} we recall the following results on the cohomology:
\begin{align}
\nonumber
H^h(V_{g,n}^0) &= 
\begin{cases}
\Q & \text{if $g=h=0$ and $n\geq 3$} \\
0 & \text{otherwise}
\end{cases}
\\
\label{equ:V2 cohom}
H^{2+h}(V_{g,n}^2) 
&=
\begin{cases}
H^2(\MM_{0,n}) & \text{if $g=0$, $h=0$, and $n\geq 4$} \\
\Q \delta_{irr} & \text{if $g=1$, $h=0$, and $n\geq 1$} \\
H^2(\MM_{0,n+2})_{\mathbb{S}_2}/\xi^*H^2(\MM_{1,n}) & \text{if $g=1$, $h=1$, and $n\geq 2$} \\
0 & \text{otherwise}
\end{cases}.
\end{align}
We recall that $H^2(\MM_{0,n})$ is spanned by the classes $\delta_A=\delta_{0,A}=\delta_{A^c}$ for $A\subset \{1,\dots,n\}$, and the subspace $\xi^*H^2(\MM_{1,n})\subset H^2(\MM_{0,n+2})_{\mathbb{S}_2}$ is spanned by classes $\delta_A$ with $\{n+1,n+2\}\subset A$ or $\{n+1,n+2\}\subset A^c$.

Note that $V_{g,n}^{11}$ has at most one non-zero term, and we obtain
\[
H^{11+h}(V^{11}_{g,n})\cong
\begin{cases}
H^{11}(\MM_{1,n+2h})_{\ss_2\wr \ss_{h}} & \text{if $g=h+1$} \\
0 & \text{otherwise}
\end{cases}.
\]
\begin{lem}
We have 
\[
H^{13+h}(V_{g,n}^{13})
=\begin{cases}
H^{13}(\MM_{1,n+2h})_{\ss_2\wr \ss_h} / J 
& \text{if $g=h+1$} \\
0 & \text{otherwise}
\end{cases}
\]
where $J=\xi^*H^{13}(\MM_{2,n+2h-2})_{\ss_2\wr \ss_{h-1}}$ for $h\geq 1$ and $J=0$ for $h=0$.
Concretely, if we decompose $J \otimes \mathbb{C} = J^{12,1} \oplus J^{1,12}$, then $J^{12,1}$ is spanned for $h>0$ by the $Z_{B\subset A}$ such that $\{s,t\}\subset A^c$, where $s,t$ correspond to the markings connected by $\xi$.
\end{lem}

\begin{proof}
We consider the term of degree $13+h$,
\[
 H^{12,1}(\MM_{g-h,n+2h})_{\ss_2\wr \ss_h} 
\]
in the complex $V_{g,n}^{12,1}$.
We index the $n+2h$ marked points in $\MM_{g-h,n+2h}$ by the set $$S_h:= \{1,\dots, n, s_1,t_1,\dots,s_h,t_h\}.$$ 
We think of self-edges being present between $s_j$ and $t_j$ (for $j=1,\dots,h$).
The coinvariants $H^{12,1}(\MM_{g-h,n+2h})_{\ss_2\wr \ss_h}$ is a quotient of 
$H^{12,1}(\Mb_{g-h,n+2h})$, and is therefore
generated by the images of $Z_{B'\subset A'} \in H^{12,1}(\Mb_{g-h,n+2h})$ with $B'\subset A'\subset S_h$. Below, we identify a subset of these generators that is a basis.

The two endpoints $s_j, t_j$ of our self-edges can each be in one of the three disjoint sets $B', A'\setminus B',$ or $(A')^c$.
Accordingly there are 6 different types of such self-edges, of which one (both endpoints in $B'$) is zero due to symmetry.

Denote the number of self-edges in the five remaining types by $h_{12},h_{13},h_{22}, h_{23}$, and $h_{33}$, where $h_{k\ell}$ is the number of self-edges between the $k$th and $\ell$th disjoint sets.
If any of $h_{22}, h_{23},$ or $h_{33}$ is $2$ or more, this element vanishes in the space of coinvariants because of the sign for permuting edges.
Let $B\subset A\subset \{1,\dots,n\}$ and numbers $h_{22},h_{23}, h_{33}\in\{0,1\}$, $h_{12},h_{13}\geq 0$ be given such that $h_{12}+h_{13}+h_{22}+ h_{23}+ h_{33}=h$, $|B|+h_{12}+h_{13}=10$ and $2(g-h-1) + |A^c|+h_{13}+h_{23}+2h_{33}\geq 2$. Here the last condition ensures that $|(A')^c| \geq 2$ if $g - h - 1 = 0$. 
Then we set $$\underline h = (h_{12},h_{13},h_{22}, h_{23}, h_{33}) \quad \mbox{ and } \quad 
Z_{ B\subset A, \underline h}
=
Z_{B\subset A, h_{12},h_{13},h_{22}, h_{23}, h_{33}} = Z_{B'\subset A'},
$$ 
with $B':=B\sqcup \{s_1,\dots, s_{h_{12}+h_{13}}\}$ and 
$$
A':= (A\cup B') \sqcup \{t_1,\dots,t_{h_{12}}\}
\sqcup \underbrace{\{ s_{h_{12}+h_{13}+1}, t_{h_{12}+h_{13}+1} \}}_{\text{if $h_{22}=1$}} 
\sqcup 
\{s_{h_{12}+h_{13}+h_{22}+1},\dots,  s_{h_{12}+h_{13}+h_{22}+h_{23}}\}.
$$ 
Then the $Z_{B\subset A, \underline h}$ form a basis 
for $H^{12,1}(\MM_{g-h,n+2h})_{\ss_2\wr \ss_h}$
if we require in addition the following condition coming from \cite[Lemma~4.11]{CLPW}:
\begin{itemize}
    \item If $|(A')^c| = 2$ and $g=h+1$, then $\min((A')^c)\in \{1,\dots,n\}$ and $\min(A')^c<\min(B')$, where we order the set $S_h= \{1,\dots, n, s_1,t_1,\dots,s_h,t_h\}$ as indicated. 
\end{itemize}
Note that in the case $A=\{1,\dots,n\}$ the element is zero modulo the relations.

In our basis, by \cite[Lemma~4.3]{CLPW}, the differential $\xi^*$ acts as 
\[
\xi^*(Z_{B\subset A, h_{12},h_{13},h_{22}, h_{23}, h_{33}})
=
\begin{cases}
0 &\text{if $g=h+1$ or $h_{33}=1$} \\
Z_{B\subset A, h_{12},h_{13},h_{22}, h_{23}, 1} &\text{if $g>h+1$ and $h_{33}=0$.} 
\end{cases}
\]
The element on the right-hand side (if not $0$) is again a basis element of 
\[H^{12,1}(\MM_{g-h-1,n+2h+2})_{\ss_2\wr \ss_h} \subset V_{g,n}^{12,1}\] except when $g=h+2$ and $|A^c|=h_{13}=h_{23}=h_{33}=0$. 
In the latter case we have 
\begin{equation}\label{equ:Z tadpoles}
Z_{B\subset A, h_{12},h_{13},h_{22}, h_{23}, 1}
=\begin{cases}
    0 & \text{if $B=\emptyset$} \\
    \pm 2
    Z_{(B\setminus i)\subset A, h_{12},h_{13}+1,h_{22}, h_{23}, 0} & \text{if $B \neq \emptyset$}
\end{cases}
\end{equation}
for $i:=\min(B)$, using the genus $1$ relations as in \cite[Lemma~4.6]{CLPW}. 
Computing cohomology, the basis elements hence cancel in pairs, except for the following:
\begin{itemize}
\item Elements $Z_{B\subset A, h_{12},h_{13},h_{22}, h_{23}, 0}$ in the case that $g=h+1$, with $|A^c|+h_{13}+h_{23}\geq 3$.
\item Elements 
$Z_{B\subset A, h_{12},0,h_{22}, h_{23}, 0}$ in the case that $g=h+1$, with $A^c\neq \emptyset$, $|A^c|+h_{23}=2$ and $\min(A^c)<\min(B)$.
\end{itemize}
In particular, these all live in the top degree $13+g-1$ of the complex $V^{13}_{g,n}$ so that the result of the lemma follows.
\end{proof}

\begin{rem}\label{rem:graph vanishing}
    Note that \eqref{equ:Z tadpoles} in particular implies that the element $Z_{(B\setminus i)\subset A, h_{12},h_{13},h_{22}, h_{23}, 0}$ with $h_{13}>0$ is trivial in the cohomology of $V^{13}_{g,n}$.
    
    Similarly, consider the definition of $\bGK^{12,1}_{g,n}$ and Lemma \ref{lem:Ign 2} above.
    Let $\Gamma\in \bGK^{12,1}_{g,n}$ be a graph that has a weight 13 vertex of the form
    \begin{equation}\label{equ:tp gr}
    \begin{tikzpicture}
        \node[ext] (v) at (0,0) {};
        \node[int] (w) at (0.7,0) {};
        \node (u) at (1.2,0) {};
        \draw (v) edge[crossed] (w) edge[->-, bend left] (w) edge +(-.3,0) edge +(-.3,.3) edge +(-.3,-.30)
        (w) edge (u);
    \end{tikzpicture}
    \end{equation}
    i.e., there is at least one self-edge with one endpoint in $B'$, the other in $(A')^c$, and we have $|(A')^c|=2$.
   The relations in $H^{12,1}(\MM_{1,n})$ are determined by \cite[Lemma 4.6]{CLPW} (for $n = 12$) and the paragraph following it (for $n \geq 12$). Using these relations, we have that 
    \begin{equation} \label{rel1}
    \begin{tikzpicture}
        \node[ext] (v) at (0,0) {};
        \node[int] (w) at (0.7,0) {};
        \node (u) at (1.4,0) {$\scriptstyle \cdots$};
        \draw (v) edge[crossed] (w) edge[->-, bend left] (w) edge +(-.3,0) edge +(-.3,.3) edge +(-.3,-.30)
        (w) edge (u);
    \end{tikzpicture}
    =- \frac 12 
        \begin{tikzpicture}
        \node[ext] (v) at (0,0) {};
        \node[int] (w) at (0.7,0) {};
        \node (u) at (0.3,0.7) {$\scriptstyle \cdots$};
        \draw (v) edge[crossed] (w) edge +(-.3,0) edge +(-.3,.3) edge +(-.3,-.30)
        (w) edge[loop] (w)
        (v) edge[->-] (u);
    \end{tikzpicture}
    =
    -
    \frac 1{24}
        \begin{tikzpicture}
        \node[ext] (v) at (0,0) {};
        \node[int] (w) at (0.7,0) {};
        \node (u) at (0.3,0.7) {$\scriptstyle \cdots$};
        \draw (v) edge (w) edge +(-.3,0) edge +(-.3,.3) edge +(-.3,-.30)
        (w) edge[loop, crossed] (w)
        (v) edge[->-] (u);
    \end{tikzpicture}
    \end{equation}
    modulo graphs of the form in (7$'$), so the above equality holds in $\bGK_{g,n}^{12,1}$.
    In particular, generators of $\bGK_{g,n}^{12,1}$ of the form \eqref{equ:tp gr} can be replaced by generators of the form given by the right-hand graph in the previous equation.
    Furthermore, note that in the case of two marked edges parallel to the crossed edge (as in the picture below) the graph is in fact zero in $\bGK_{g,n}^{12,1}$ 
    by the relations in $H^{12,1}(\Mb_{1,n})$ and symmetry:
    \begin{equation} \label{tripodzero}
    \begin{tikzpicture}
        \node[ext] (v) at (0,0) {};
        \node[int] (w) at (0.7,0) {};
        \draw (v) edge[crossed] (w) edge[->-, bend left] (w) edge[->-, bend right] (w) edge +(-.3,0) edge +(-.3,-.15) edge +(-.3,-.30);
    \end{tikzpicture}
    =- \frac 12 
        \begin{tikzpicture}
        \node[ext] (v) at (0,0) {};
        \node[int] (w) at (0.7,0) {};
        \draw (v) edge[crossed] (w) edge +(-.3,0) edge +(-.3,-.15) edge +(-.3,-.3)
        (w) edge[loop] (w);
        \draw (v) edge[loop] (v)
        (v) edge[->] +(50:.2)
        (v) edge[->] +(140:.2)
        ;
    \end{tikzpicture}
    =
0.
    \end{equation}
\end{rem}

\subsection{Proof of Proposition \ref{prop:bGK}}\label{sec:prop bGK proof}
We equip both the domain and target of the projection 
\[
\pi \colon \GK_{g,n}^{13} \to \GK_{g,n}^{13}/I_{g,n} =: \bGK_{g,n}^{13}
\]
with the descending filtration on the number of vertices in the generating graphs. The associated spectral sequences automatically converge to the cohomologies by finite dimensionality.
Hence it is sufficient to show that the morphism $\pi$ induces a quasi-isomorphism on the $E_1$-pages of the spectral sequences.
First consider the $E^0$-page:
The associated graded differential on $\GK_{g,n}^{13}$ is given by $d_\xi$, while on $\GK_{g,n}^{13}/ I_{g,n}$ it is zero. 
The complex $(\GK_{g,n}^{13}, d_\xi)$ is identified with a sum of direct summands of tensor products of one-vertex complexes. More precisely, we may write 
\[
(\GK_{g,n}^{13}, d_\xi)
\cong 
\Big(\bigoplus_{\Gamma} 
\otimes_{v\in V(\Gamma)}
V_{g_v, n_v}^{k_v}
\otimes \Q[-1]^{\otimes |E(\Gamma)|}
\Big)_{\Aut(\Gamma)}
\]
where the sum is over isomorphism classes of stable graphs without self-loops of genus $g$, with $n$ external legs and with weighted vertices of total weight $13$, and the $V_{g_v, n_v}^{k_v}$ are the one-vertex complexes of the preceding section.
Since taking cohomology commutes with taking coinvariants under finite group actions, as well as sums and tensor products, we find that the $E_1$-page of our spectral sequence is 
\[
E_1:=
H(\GK_{g,n}^{13}, d_\xi)
\cong 
\Big(\bigoplus_{\Gamma} 
\otimes_{v\in V(\Gamma)}
H(V_{g_v, n_v}^{k_v})
\otimes \Q[-1]^{\otimes |E(\Gamma)|}
\Big)_{\Aut(\Gamma)}.
\]
Comparing the cohomologies of the one-vertex complexes $H(V_{g_v, n_v}^{k_v})$ computed in the previous section with the definition of $\bGK_{g,n}^{13}$, we see that the induced map $E_1\to \bGK_{g,n}^{13}$ is almost an isomorphism. The only difference is that by \eqref{equ:V2 cohom} the cohomology of $V_{1, m}^{2}$ is non-trivial in degree 2 for every $m\geq 1$ (represented by $\delta_{irr}$), whereas in the definition of $\bGK_{g,n}^{13}$ we set all graphs with a genus-1-weight-$2$ vertex of valence $\geq 1$ or not adjacent to the weight 11 vertex to zero.  
To address this difference one proceeds again as in  \cite{PayneWillwacher21} and considers a second, nested spectral sequence such that the differential on the associated graded splits off the weight 2 vertex and creates a univalent vertex.
\[
  s\colon
  \begin{tikzpicture}
    \node[int] (v) at (0,0) {};
    \draw (v) edge +(-.5,-.5) edge +(0,-.5) edge +(.5,-.5)
    edge[loop, crossed] (v);
  \end{tikzpicture}
\mapsto 
\begin{tikzpicture}
  \node[int] (v) at (0,0.5) {};
  \node[int] (w) at (0,0) {};
  \draw (w) edge +(-.5,-.5) edge +(0,-.5) edge +(.5,-.5) edge (v)
  (v) edge[loop, crossed] (v);
\end{tikzpicture},
\]
This differential has a homotopy that contracts the complex $E_1$ to the subcomplex spanned by graphs such that the weight 2 vertex (if present) is univalent and connected to the weight 11 vertex. Hence the map $E_1\to \bGK_{g,n}^{13}$ is a quasi-isomorphism and we are done.
\hfill\qed 

\subsection{The differential on \texorpdfstring{$\bGK_{g,n}^{12,1}$}{bGK}}
\label{sec:differential}
Recall from Section \ref{sec:feyn recollection} that the differential on $\GK_{g,n}$ has two terms $d=d_\vartheta+d_\xi$.
This differential descends to a differential on the complex $\bGK_{g,n}^{12,1}$  whose combinatorial form we shall now describe.
First, the piece $d_\xi$ is zero on $\bGK_{g,n}^{12,1}$, essentially by construction.
Graphically the piece $d_\vartheta$ acts by splitting vertices. We shall describe the action on each type of vertex.
\begin{itemize}
    \item Weight 0 vertices: Here $d_\vartheta$ acts by replacing the vertex with two vertices, summing over all ways of attaching the incident half-edges:
\[
\begin{tikzpicture}
    \node[int] (v) at (0,0) {};
    \draw (v) edge +(-.3,-.3) edge +(-.3,.3) edge +(-.3,0)              edge +(.3,-.3) edge +(.3,0) edge +(.3,.3);
\end{tikzpicture}
\mapsto\sum\,
\begin{tikzpicture}
    \node[int] (v) at (0,0) {};
    \node[int] (w) at (0.5,0) {};
    \draw (v) edge +(-.3,-.3) edge +(-.3,.3) edge +(-.3,0) edge (w) 
    (w) edge +(.3,-.3) edge +(.3,0) edge +(.3,.3);
\end{tikzpicture}
\]
    \item Weight 2 vertex: 
    Here it is helpful to represent the decoration at the weight 2 vertex by drawing two vertices connected by a crossed edge, see \eqref{equ:delta crossed}. Then the action of the differential has three terms, see also \cite[Section 3.3]{PayneWillwacher21}.
\begin{equation}\label{equ:delta split delta}
\begin{aligned}
\begin{tikzpicture}[scale=.9]
  \node[int] (v) at (0,0){};
  \node[int] (w) at (-.7,0){};
  \draw (v) edge[crossed] (w) edge +(0:.5) edge +(60:.5) edge +(-60:.5) 
  (w) edge +(120:.5) edge +(180:.5) edge +(-120:.5);
\end{tikzpicture}
&\mapsto 
\sum_{b'+b''=b}\sum
\begin{tikzpicture}[scale=.9]
  \node[int] (v0) at (0.7,0){};
  \node[int] (v) at (0,0){};
  \node[int] (w) at (-.7,0){};
  \draw (v) edge[crossed] (w) edge (v0) edge +(-60:.5) 
  (v0) edge +(0:.5) edge +(60:.5) 
  (w) edge +(120:.5) edge +(180:.5) edge +(-120:.5);
\end{tikzpicture}
+\sum_{a'+a''=a}\sum
\begin{tikzpicture}[scale=.9]
  \node[int] (w0) at (-1.4,0){};
  \node[int] (v) at (0,0){};
  \node[int] (w) at (-.7,0){};
  \draw (v) edge[crossed] (w)  edge +(-60:.5) edge +(0:.5) edge +(60:.5) 
  (w) edge (w0) edge +(120:.5) 
  (w0) edge +(180:.5) edge +(-120:.5);
\end{tikzpicture}
\\&\quad -
\begin{tikzpicture}[scale=.9]
  \node[int] (v) at (0,0){};
  \node[int,label=0:{$\scriptstyle \psi$}] (w) at (-1,0){};
  \draw (v) edge +(0:.5) edge +(60:.5) edge +(-60:.5) 
  (w) edge (v) edge +(120:.5) edge +(180:.5) edge +(-120:.5);
\end{tikzpicture}
-
\begin{tikzpicture}[scale=1]
  \node[int,label=180:{$\scriptstyle \psi$}] (v) at (0,0){};
  \node[int] (w) at (-1,0){};
  \draw (v) edge (w) edge +(0:.5) edge +(60:.5) edge +(-60:.5) 
  (w) edge +(120:.5) edge +(180:.5) edge +(-120:.5);
\end{tikzpicture}
\end{aligned}
\end{equation}
The $\psi$ in the second line indicates a decorations with a $\psi$-class at the marking pointing to the other vertex. 
The $\psi$ class here can (and should) again be replaced by a decoration by boundary classes $\delta_S$ using the relations on $H^*(\MM_{0,n})$.
    \item Weight 11 vertex:
    A weight 11 vertex with decoration $\omega_A$ is split by pushing off a subset of the half-edges not in $A$ and at most one half-edge in $A$:
    \begin{align} \label{wt11split}
\begin{tikzpicture}
        \node[ext] (v) at (0,0) {};
        \draw (v) edge[->-] +(-.3,-.3)  edge[->-] +(-.3,0) edge +(-.3,.3) edge[->-] +(.3,-.3)  edge +(.3,0) edge +(.3,.3);
        \end{tikzpicture}
        &\mapsto\sum
        \begin{tikzpicture}[baseline=-.65ex]
        \node[ext] (v) at (0,0) {};
        \node[int] (w) at (0.5,0) {};
        \draw (v) edge (w) (v) edge[->-] +(-.3,-.3)  edge[->-] +(-.3,0) edge +(-.3,.3) edge[->-] +(.3,-.3)
         (w)   edge +(.3,0) edge +(.3,.3);
        \end{tikzpicture}
        +\sum 
        \begin{tikzpicture}[baseline=-.65ex]
        \node[ext] (v) at (0,0) {};
        \node[int] (w) at (0.5,0) {};
        \draw (v) edge[->-] (w) (v) edge[->-] +(-.3,-.3)  edge[->-] +(-.3,0) edge +(-.3,.3)
            (w) edge +(.3,-.3)  edge +(.3,0) edge +(.3,.3);
        \end{tikzpicture} \quad .
\end{align}
Note that there must be eleven marked half-edges at the vertex, but we only draw 3 for simplicity of the picture.
    \item Weight 13 vertex:
    It is again helpful to represent the decoration (by $Z_{B\subset A}$, say) at the weight 13 vertex by drawing two vertices connected by a crossed edge.
    The action of the differential $d_\vartheta$ is then determined by the pullback formulas \cite[Lemmas 4.4 and 4.5]{CLPW}. Graphically, this reads:
    \begin{equation}\label{equ:split 13}
    \begin{aligned}
\begin{tikzpicture}
    \node[ext] (v) at (0,0) {};
    \node[int] (w) at (0.7,0) {};
    \draw (v) edge +(-.5,.5) edge +(-.5,.3) edge[->-] +(-.5,.1) edge[crossed] (w) 
    (w) edge +(.5,-.5) edge +(.5,0) edge +(.5,.5);
    \draw (v) edge[->-] +(-.5,-.3) edge[->-] +(-.5,-.5);
    \draw [decorate,decoration={brace,amplitude=5pt}]
  (-1.2,-.5) -- (-1.2,.5) node[midway,xshift=-1em]{$A$};
    \draw [decorate,decoration={brace,amplitude=5pt}]
  (-.7,-.5) -- (-.7,.1) node[midway,xshift=-.7em]{$\scriptstyle B$};
\end{tikzpicture}
\mapsto
&
\sum 
\begin{tikzpicture}
    \node[ext] (v) at (0,0) {};
    \node[int] (vv) at (-.3,.3) {};
    \node[int] (w) at (0.7,0) {};
    \draw (v)  edge[->-] +(-.5,.1) edge[crossed] (w) 
    (w) edge +(.5,-.5) edge +(.5,0) edge +(.5,.5);
    \draw (v) edge[->-] +(-.5,-.3) edge[->-] +(-.5,-.5) edge (vv);
    \draw (vv) edge +(-.5,.5) edge +(-.5,.3);
\end{tikzpicture}
+
\sum 
\begin{tikzpicture}
    \node[ext] (v) at (0,0) {};
    \node[int] (vv) at (-.4,.1) {};
    \node[int] (w) at (0.7,0) {};
    \draw (v)  edge[->-] +(-.5,-.3) edge[crossed] (w) edge +(-.5,.5) 
    (w) edge +(.5,-.5) edge +(.5,0) edge +(.5,.5);
    \draw (v)  edge[->-] +(-.5,-.5) edge[->-] (vv);
    \draw (vv) edge +(-.5,.3) edge +(-.5,-.3);
\end{tikzpicture}
+
\sum 
\begin{tikzpicture}
    \node[ext] (v) at (0,0) {};
    \node[int] (vv) at (1,.3) {};
    \node[int] (w) at (0.7,0) {};
    \draw (v) edge +(-.5,.5) edge +(-.5,.3) edge[->-] +(-.5,.1) edge[crossed] (w) 
    (w) edge +(.5,-.5) ;
    \draw (v) edge[->-] +(-.5,-.3) edge[->-] +(-.5,-.5);
    \draw (vv) edge (w) edge +(.5,0) edge +(.5,.5);
\end{tikzpicture}
   \\
   & -
    \begin{tikzpicture}
    \node[ext,label=0:{$\scriptstyle \psi$}] (v) at (0,0) {};
    \node[int] (w) at (1,0) {};
    \draw (v) edge +(-.5,.5) edge +(-.5,.3) edge[->-] +(-.5,.1) edge[->-] (w) 
    (w) edge +(.5,-.5) edge +(.5,0) edge +(.5,.5);
    \draw (v) edge[->-] +(-.5,-.3) edge[->-] +(-.5,-.5);
\end{tikzpicture}
    -
    \begin{tikzpicture}
    \node[ext] (v) at (0,0) {};
    \node[int,label=180:{$\scriptstyle \psi$}] (w) at (1,0) {};
    \draw (v) edge +(-.5,.5) edge +(-.5,.3) edge[->-] +(-.5,.1) edge[->-] (w) 
    (w) edge +(.5,-.5) edge +(.5,0) edge +(.5,.5);
    \draw (v) edge[->-] +(-.5,-.3) edge[->-] +(-.5,-.5);
\end{tikzpicture}
+
\sum 
    \begin{tikzpicture}
    \node[ext] (v) at (0,0) {};
    \node[int] (vv) at (0.7,0) {};
    \node[int] (w) at (1.4,0) {};
    \draw (v)  edge +(-.5,.3) edge[->-] +(-.5,.1) edge[->-] (vv) (vv) edge[crossed] (w) edge +(-.5,.5)
    (w) edge +(.5,-.5) edge +(.5,0) edge +(.5,.5);
    \draw (v) edge[->-] +(-.5,-.3) edge[->-] +(-.5,-.5);
\end{tikzpicture}
\end{aligned}
    \end{equation}
    Again, the $\psi$-class decoration in the second line can be replaced by a boundary class using the relations between $\psi$-classes and boundary classes in genus 0 and 1.
\end{itemize}

\newcommand{\myB}{\bGK^{12,1}}
\section{Low excess computations in weight 13}\label{sec:low excess13}
\subsection{Blown-up pictures of graphs}\label{sec:blownup}
Generators of $\myB_{g,n}$ are connected graphs, with exactly one vertex of weight 11 or 13, and for the decoration of this vertex we may use the graphical notation of the previous section.
Here we perform explicit computations of the cohomology of $\myB_{g,n}$ for the pairs $(g,n)$ such that $3g + 2n$ is at most 27. In particular, we prove Proposition~\ref{prop:wt 13 vanishing} and Theorem~\ref{thm:lowexc13}. To this end, we need to enumerate all contributing graphs, and this is notationally easier if we draw our graphs differently, using the blown-up picture as in \cite{PayneWillwacher21, PayneWillwacher24}.
More precisely, for a weight 11 vertex decorated by $\omega_A$, we remove the genus $1$ vertex, leaving the edges incident to $v$ as external legs. Of those, we mark the ones in $A$ with a symbol $\omega$ and the others by $\epsilon$. 
\[
\begin{tikzpicture}
    \node[ext] (v) at (0,0) {};
    \node at (.5, .1) {$\scriptscriptstyle \vdots$};
    \draw (v) edge +(-.5,-.5) edge +(-.5,0) edge +(-.5,+.5) 
    edge[->-] +(.5,+.5) edge[->-] +(.5,+.3) edge[->-] +(.5,-.3) edge[->-] +(.5,-.5)
    ;
    \draw [decorate,decoration={brace,amplitude=5pt}]
  (.8,.5) -- (.8,-.5) node[midway,xshift=1em]{$A$};
  \end{tikzpicture}
\quad \to \quad
\begin{tikzpicture}
\node (e1) at (0,0) {$\epsilon$};
\node  at (.6,0) {$\cdots$};
\node (e2) at (1.2,0) {$\epsilon$};
\node (o1) at (1.7,0) {$\omega$};
\node  at (2.3,0) {$\cdots$};
\node (o2) at (2.9,0) {$\omega$};
\draw (e1) edge +(0,.5) (e2) edge +(0,.5) (o1) edge +(0,.5) (o2) edge +(0,.5) ;
\end{tikzpicture}
\]
Similarly, a weight $13$ vertex decorated by $Z_{B\subset A}$ is represented as in the following picture 
\[
\begin{tikzpicture}
    \node[ext] (v) at (0,0) {};
    \node[int] (w) at (0.7,0) {};
    \draw (v) edge +(-.5,.5) edge +(-.5,.3) edge[->-] +(-.5,.1) edge[crossed] (w) 
    (w) edge +(.5,-.5) edge +(.5,0) edge +(.5,.5);
    \draw (v) edge[->-] +(-.5,-.3) edge[->-] +(-.5,-.5);
    \draw [decorate,decoration={brace,amplitude=5pt}]
  (-1.2,-.5) -- (-1.2,.5) node[midway,xshift=-1em]{$A$};
    \draw [decorate,decoration={brace,amplitude=5pt}]
  (-.7,-.5) -- (-.7,.1) node[midway,xshift=-.7em]{$\scriptstyle B$};
\end{tikzpicture}
\to 
\begin{tikzpicture}
\node (e1) at (0,0) {$\epsilon$};
\node  at (.6,0) {$\cdots$};
\node (e2) at (1.2,0) {$\epsilon$};
\node (o1) at (1.7,0) {$\omega$};
\node  at (2.3,0) {$\cdots$};
\node (o2) at (2.9,0) {$\omega$};
\node (o3) at (3.4,0) {$\omega$};
\node[int] (i) at (3.4,.7) {};
\draw (e1) edge +(0,.5) (e2) edge +(0,.5) (o1) edge +(0,.5) (o2) edge +(0,.5) ;
\draw (i) edge[crossed] (o3) edge +(0,.5) edge +(0.5,.5) edge +(-0.5,.5);
\draw [decorate,decoration={brace,amplitude=5pt}]
  (3.1,-.3) -- (1.5,-.3) node[midway,yshift=-.9em]{$\scriptstyle B$};
\draw [decorate,decoration={brace,amplitude=5pt}]
  (3.1,-.9) -- (-.2,-.9) node[midway,yshift=-.9em]{$\scriptstyle A$};
\end{tikzpicture}
\]
We emphasize that this is merely a different way of drawing graphs.
Furthermore, we remind the reader that due to the relations on the decorations, our graph complex $\myB_{g,n}$ is not fully combinatorial, i.e., there are relations between some graphs.

\subsection{Excess}
We define the excess
\[
E(g,n) := 3g +2n - 25.
\]
We will compute the cohomology of $\myB_{g,n}$ for all pairs $g,n$ such that $E(g,n)\leq 2$. 

To this end, let us write each generator $\Gamma$ of $\myB_{g,n}$ as a union of its blown-up components:
\[
\Gamma = C_1\cup\cdots \cup C_k.
\]
Let $g_i$ be the contribution of $C_i$ to the genus of $\Gamma$.  More precisely,
\[
g_i = h^1(C_i) + \# \epsilon + \#\omega-1,
\]
i.e., the loop order of $C_i$ plus the number of its $\epsilon$ and $\omega$ labeled legs minus one.  Then we define the excess of $C_i$ as 
\begin{equation}\label{equ:exc comp}
E(C_i)
:= 3g_i +2n_i - 2\# \omega
=
3h^1(C_i) + 3\# \epsilon + \#\omega+2n_i -3. 
\end{equation}

\begin{lem}\label{lem:excess13}
The excess is additive over blown-up components, in the sense that for any graph $\Gamma=C_1\cup\cdots \cup C_k\in \myB_{g,n}$ we have
\begin{equation}\label{equ:E additive13}
E(g,n) = E(C_1)+ \cdots + E(C_k),  
\end{equation}
and the excess of each blown-up component is non-negative.  
\end{lem}
\begin{proof}
For the first statement note that $g= 1+\sum_{i=1}^kg_i$, and the total number of $\omega$-legs must be 11.
The second statement follows as in \cite[Lemma 4.2]{PayneWillwacher24}.
\end{proof}

The excess provides a measure of the combinatorial complexity of the relevant graphs. In particular, the additivity property of the above lemma helps us
to organize our calculations. By enumerating all blown up components of low excess, it becomes possible to compute $H^k(\myB_{g,n})$ by hand for all $g, n$ with $E(g, n) \leq 2$.

\subsection{Excess 0 and proof of Proposition \ref{prop:wt 13 vanishing}}
Blown-up components of excess 0 must have loop order 0 and at most 3 legs by \eqref{equ:exc comp} and are hence easily seen to be one of the following forms:
\[
  \begin{tikzpicture}
    \node (v) at (0,0) {$\omega$};
    \node (w) at (1,0) {$j$};
    \draw (v) edge (w);
  \end{tikzpicture}
      \quad\quad
      \text{or} 
      \quad\quad
    \begin{tikzpicture}
        \node[int] (i) at (0,.5) {};
        \node (v1) at (-.5,-.2) {$\omega$};
        \node (v2) at (0,-.2) {$\omega$};
        \node (v3) at (.5,-.2) {$\omega$};
      \draw (i) edge (v1) edge (v2) edge (v3);
      \end{tikzpicture} 
      \quad\quad 
      \text{or} 
      \quad\quad 
      \begin{tikzpicture}
        \node[int] (i) at (0,.5) {};
        \node (v1) at (-.5,-.2) {$\omega$};
        \node (v2) at (0,-.2) {$\omega$};
        \node (v3) at (.5,-.2) {$\omega$};
      \draw (i) edge[crossed] (v1) edge (v2) edge (v3);
      \end{tikzpicture} 
      .
\]
Using the conventions introduced in \S \ref{pics},
each generator of $\myB_{g,n}$ must have exactly one marked edge. 
Graphs with the third component in the list above are however zero in $\myB_{g,n}$ by \eqref{tripodzero}.
This means that every generator $\Gamma\in \myB_{g,n}$ must have a blown-up component of excess at least one, and we arrive at the following corollary of Lemma \ref{lem:excess13}, that is also stated as Proposition \ref{prop:wt 13 vanishing} in the introduction.

\begin{cor} \label{cor:Eneg13}
If $E(g,n) <1$ then $
\gr^W_{13} H^*_c(\M_{g,n}) =0.$
\end{cor}

\subsection{Excess 1} \label{sec:exc1-13}
We next consider the case of excess one. In this case generators of $\myB_{g,n}$ must be unions of blown-up components of excess zero as above, and one blown-up component of excess one, containing a crossed edge.
The combinatorially possible excess 1 components with a crossed edge are:
\begin{equation}\label{equ:exc 1 special}
\begin{aligned}
    \begin{tikzpicture}
        \node[int] (i) at (0,.5) {};
        \node (v1) at (-.5,-.2) {$\omega$};
        \node (v2) at (0,-.2) {$\omega$};
        \node (v3) at (.5,-.2) {$j$};
      \draw (i) edge[crossed] (v1) edge (v2) edge (v3);
      \end{tikzpicture}
      =-\frac 1{24} \bigg(
            \begin{tikzpicture}
        \node (i) at (0,0) {$\omega$};
        \node (v2) at (.7,0) {$j$};
      \draw (i) edge (v2) ;
      \end{tikzpicture}
        \begin{tikzpicture}
        \node[int] (i) at (0,.5) {};
        \node (v2) at (0,-.2) {$\omega$};
      \draw (i) edge[crossed, loop] (i) edge (v2) ;
      \end{tikzpicture}
      \bigg)
      \quad\text{or}\quad 
      \begin{tikzpicture}
        \node[int] (i) at (0,.5) {};
        \node (v2) at (0,-.2) {$\omega$};
      \draw (i) edge[crossed, loop] (i) edge (v2) ;
      \end{tikzpicture}
      \quad\text{or}\quad 
      \begin{tikzpicture}
        \node[int] (i) at (0,.5) {};
        \node[int] (j) at (1,.5) {};
        \node (v1) at (-.5,-.2) {$\omega$};
        \node (v2) at (0,-.2) {$\omega$};
        \node (v3) at (1,-.2) {$\omega$};
        \node (v4) at (1.5,-.2) {$\omega$};
      \draw (i) edge (v1) edge (v2) edge[crossed] (j) (j) edge (v3) edge (v4);
      \end{tikzpicture}
      \\
      \quad\text{or}\quad 
      \begin{tikzpicture}
        \node[int] (i) at (0,.5) {};
        \node[int] (j) at (1,.5) {};
        \node (v1) at (-.5,-.2) {$\omega$};
        \node (v2) at (0,-.2) {$\omega$};
        \node (v3) at (1,-.2) {$\omega$};
        \node (v4) at (1.5,-.2) {$\omega$};
      \draw (i) edge[crossed] (v1) edge (v2) edge (j) (j) edge (v3) edge (v4);
      \end{tikzpicture}=-\frac1{24}\bigg(
      \begin{tikzpicture}
        \node[int] (i) at (0,.5) {};
        \node (v2) at (0,-.2) {$\omega$};
      \draw (i) edge[crossed, loop] (i) edge (v2) ;
      \end{tikzpicture}
        \begin{tikzpicture}
    \node[int] (i) at (0,.5) {};
    \node (v1) at (-.5,-.2) {$\omega$};
    \node (v2) at (0,-.2) {$\omega$};
    \node (v3) at (.5,-.2) {$\omega$};
  \draw (i) edge (v1) edge (v2) edge (v3);
  \end{tikzpicture}
      \bigg)
      \quad\text{or}\quad 
      \begin{tikzpicture}
        \node[int] (i) at (0,.5) {};
        \node (v1) at (-.7,-.2) {$\omega$};
        \node (v2) at (-.2,-.2) {$\omega$};
        \node (v3) at (.2,-.2) {$\omega$};
        \node (v4) at (.7,-.2) {$\omega$};
      \draw (i) edge[crossed] (v1) edge (v2) edge (v3) edge (v4);
      \end{tikzpicture}
      .
\end{aligned}
\end{equation}
The first and the fourth component can be replaced by other terms using the relations in Remark \ref{rem:graph vanishing}.
Thus, if $E(g,n) = 1$, we have three possible generators for $\myB_{g,n}$: 
\[
  \Gamma^{(0)}_{irr}=
  \begin{tikzpicture}
    \node (v) at (0,0) {$\omega$};
    \node (w) at (1,0) {$1$};
    \draw (v) edge (w);
  \end{tikzpicture}
  \cdots 
  \begin{tikzpicture}
    \node (v) at (0,0) {$\omega$};
    \node (w) at (1,0) {$n$};
    \draw (v) edge (w);
  \end{tikzpicture}
  \begin{tikzpicture}
    \node[int] (i) at (0,.5) {};
    \node (v1) at (-.5,-.2) {$\omega$};
    \node (v2) at (0,-.2) {$\omega$};
    \node (v3) at (.5,-.2) {$\omega$};
  \draw (i) edge (v1) edge (v2) edge (v3);
  \end{tikzpicture}
  \cdots 
  \begin{tikzpicture}
    \node[int] (i) at (0,.5) {};
    \node (v1) at (-.5,-.2) {$\omega$};
    \node (v2) at (0,-.2) {$\omega$};
    \node (v3) at (.5,-.2) {$\omega$};
  \draw (i) edge (v1) edge (v2) edge (v3);
  \end{tikzpicture}
  \begin{tikzpicture}
    \node[int] (i) at (0,.5) {};
    \node (v2) at (0,-.2) {$\omega$};
  \draw (i) edge[crossed, loop] (i) edge (v2) ;
  \end{tikzpicture}
\]
\[
  \Gamma^{(0)}_{\delta}=
  \begin{tikzpicture}
    \node (v) at (0,0) {$\omega$};
    \node (w) at (1,0) {$1$};
    \draw (v) edge (w);
  \end{tikzpicture}
  \cdots 
  \begin{tikzpicture}
    \node (v) at (0,0) {$\omega$};
    \node (w) at (1,0) {$n$};
    \draw (v) edge (w);
  \end{tikzpicture}
  \begin{tikzpicture}
    \node[int] (i) at (0,.5) {};
    \node (v1) at (-.5,-.2) {$\omega$};
    \node (v2) at (0,-.2) {$\omega$};
    \node (v3) at (.5,-.2) {$\omega$};
  \draw (i) edge (v1) edge (v2) edge (v3);
  \end{tikzpicture}
  \cdots 
  \begin{tikzpicture}
    \node[int] (i) at (0,.5) {};
    \node (v1) at (-.5,-.2) {$\omega$};
    \node (v2) at (0,-.2) {$\omega$};
    \node (v3) at (.5,-.2) {$\omega$};
  \draw (i) edge (v1) edge (v2) edge (v3);
  \end{tikzpicture}
  \begin{tikzpicture}
    \node[int] (i) at (0,.5) {};
    \node[int] (j) at (1,.5) {};
    \node (v1) at (-.5,-.2) {$\omega$};
    \node (v2) at (0,-.2) {$\omega$};
    \node (v3) at (1,-.2) {$\omega$};
    \node (v4) at (1.5,-.2) {$\omega$};
  \draw (i) edge (v1) edge (v2) edge[crossed] (j) (j) edge (v3) edge (v4);
  \end{tikzpicture}
\]
\[
  \Gamma^{(0)}_{s}=
  \begin{tikzpicture}
    \node (v) at (0,0) {$\omega$};
    \node (w) at (1,0) {$1$};
    \draw (v) edge (w);
  \end{tikzpicture}
  \cdots 
  \begin{tikzpicture}
    \node (v) at (0,0) {$\omega$};
    \node (w) at (1,0) {$n$};
    \draw (v) edge (w);
  \end{tikzpicture}
  \begin{tikzpicture}
    \node[int] (i) at (0,.5) {};
    \node (v1) at (-.5,-.2) {$\omega$};
    \node (v2) at (0,-.2) {$\omega$};
    \node (v3) at (.5,-.2) {$\omega$};
  \draw (i) edge (v1) edge (v2) edge (v3);
  \end{tikzpicture}
  \cdots 
  \begin{tikzpicture}
    \node[int] (i) at (0,.5) {};
    \node (v1) at (-.5,-.2) {$\omega$};
    \node (v2) at (0,-.2) {$\omega$};
    \node (v3) at (.5,-.2) {$\omega$};
  \draw (i) edge (v1) edge (v2) edge (v3);
  \end{tikzpicture}
  \begin{tikzpicture}
    \node[int] (i) at (0,.5) {};
    \node (v1) at (-.7,-.2) {$\omega$};
    \node (v2) at (-.2,-.2) {$\omega$};
    \node (v3) at (.2,-.2) {$\omega$};
    \node (v4) at (.7,-.2) {$\omega$};
  \draw (i) edge[crossed] (v1) edge (v2) edge (v3) edge (v4);
  \end{tikzpicture}
\]
In each case, the total number of $\omega$-legs must be 11.
Note that the graphs $\Gamma^{(0)}_{\delta}$ and $\Gamma^{(0)}_{s}$ only exist in genera $g\geq 4$. When $g\geq 4$ and $n$ is such that $3g+2n=26$, the three graphs above form a basis of $\myB_{g,n}$. In genus $2$, we have that $\myB_{2,10}$ is one-dimensional, spanned by $\Gamma^{(0)}_{irr}$.

The cohomological degree of the graphs $\Gamma^{(0)}_{irr}$ and $\Gamma^{(0)}_{\delta}$ above is $k=24-n$, while that of $\Gamma^{(0)}_{s}$ is $k=23-n$.
Using \eqref{equ:split 13},
one checks that the differential sends $\Gamma^{(0)}_{s}$ to a non-zero multiple of $\Gamma^{(0)}_{\delta}$ plus a multiple of $\Gamma^{(0)}_{irr}$. Note that of the six terms in \eqref{equ:split 13} only the the third and the fifth term contribute, with the third term producing a multiple of $\Gamma^{(0)}_{irr}$ (after using the relation \eqref{rel1}), and the fifth term producing $-\Gamma^{(0)}_{\delta}$ (after replacing the $\psi$ class with a boundary class). In particular, the first term in the second line of \eqref{equ:split 13} is zero in this case because there cannot be a weight 13 vertex of genus 1 and valence 11.
In the end, we retain only one-dimensional cohomology, generated by $\Gamma^{(0)}_{irr}$. We hence obtain:
\begin{align*}   H^k(\myB_{2,10}) &=
    \begin{cases}
        V_{1^{10}} & \text{for $k=14$} \\
        0 & \text{otherwise}
    \end{cases}
&
H^k(\myB_{4,7}) &=
    \begin{cases}
        V_{1^{7}} & \text{for $k=17$} \\
        0 & \text{otherwise}
    \end{cases}
\\
H^k(\myB_{6,4}) &=
\begin{cases}
    V_{1^{4}} & \text{for $k=20$} \\
    0 & \text{otherwise}
\end{cases}     
&
H^k(\myB_{8,1}) &=
\begin{cases}
    V_1 & \text{for $k=23$} \\
    0 & \text{otherwise.}
\end{cases}  
\end{align*}

\subsection{Excess 2} \label{sec:exc2-13}
A general generator $\Gamma$ of $\myB_{g,n}$ in excess 2 can have either two blown-up components of excess one each, or one component of excess 2, necessarily with a crossed edge.
As in \cite[Section 4.2]{PayneWillwacher24}, the relevant excess one components without crossed edge are 
\begin{equation}\label{equ:exc 1 trees}
  \begin{tikzpicture}
    \node (v) at (0,0) {$\omega$};
    \node (w) at (1,0) {$\epsilon$};
    \draw (v) edge (w);
  \end{tikzpicture}
\quad\text{or}\quad  
\begin{tikzpicture}
  \node[int] (i) at (0,.5) {};
  \node (v1) at (-.5,-.2) {$j$};
  \node (v2) at (0,-.2) {$\omega$};
  \node (v3) at (.5,-.2) {$\omega$};
\draw (i) edge (v1) edge (v2) edge (v3);
\end{tikzpicture}
\quad\text{or}\quad  
\begin{tikzpicture}
  \node[int] (i) at (0,.5) {};
  \node (v1) at (-.6,-.2) {$\omega$};
  \node (v2) at (-.2,-.2) {$\omega$};
  \node (v3) at (.2,-.2) {$\omega$};
  \node (v4) at (.6,-.2) {$\omega$};
\draw (i) edge (v1) edge (v2) edge (v3) edge (v4);
\end{tikzpicture}
\quad\text{or}\quad  
\begin{tikzpicture}
  \node[int] (i) at (0,.5) {};
  \node[int] (j) at (1,.5) {};
  \node (v1) at (-.5,-.2) {$\omega$};
  \node (v2) at (0,-.2) {$\omega$};
  \node (v3) at (1,-.2) {$\omega$};
  \node (v4) at (1.5,-.2) {$\omega$};
\draw (i) edge (v1) edge (v2) edge (j) (j) edge (v3) edge (v4);
\end{tikzpicture}
\, .
\end{equation}
Combining one such blown-up component and one from \eqref{equ:exc 1 special} we obtain 12 possible generators with two blown-up components of excess 1.
There is a larger set of possible excess 2 components containing a crossed edge, listed below.
\begin{align*}
 & \begin{tikzpicture}
    \node[int] (i) at (0,.5) {};
    \node[int] (j) at (1,.5) {};
    \node (v2) at (0,-.2) {$\omega$};
    \node (v3) at (1,-.2) {$\omega$};
  \draw (i) edge[bend left] (j) edge (v2) edge[crossed] (j) (j) edge (v3);
  \end{tikzpicture}
  =
    \begin{tikzpicture}
      \node[int] (i) at (0,.5) {};
      \node[int] (j) at (1,.5) {};
      \node (v1) at (-.5,-.2) {$\omega$};
      \node (v2) at (0,-.2) {$\omega$};
    \draw (i) edge (v1) edge (v2) edge[crossed] (j) (j) edge[loop] (j);
    \end{tikzpicture}
    =(\cdots)
  \quad\text{or}\quad 
  \begin{tikzpicture}
    \node[int] (i) at (0,.5) {};
    \node (v1) at (-.5,-.2) {$\omega$};
    \node (v2) at (0,-.2) {$i$};
    \node (v3) at (.5,-.2) {$j$};
  \draw (i) edge[crossed] (v1) edge (v2) edge (v3);
  \end{tikzpicture}
  \quad\text{or}\quad 
  \begin{tikzpicture}
    \node[int] (i) at (0,.5) {};
    \node[int] (j) at (1,.5) {};
    \node (v1) at (-.5,-.2) {$\omega$};
    \node (v2) at (0,-.2) {$\omega$};
    \node (v3) at (1,-.2) {$\omega$};
    \node (v4) at (1.5,-.2) {$i$};
  \draw (i) edge[crossed] (v1) edge (v2) edge (j) (j) edge (v3) edge (v4);
  \end{tikzpicture}
  =(\cdots)
  \\&
  \quad\text{or}\quad 
  \begin{tikzpicture}
    \node[int] (i) at (0,.5) {};
    \node (v1) at (-.5,-.2) {$\omega$};
    \node (v2) at (0,-.2) {$\omega$};
    \node (v3) at (.5,-.2) {$\epsilon$};
  \draw (i) edge[crossed] (v1) edge (v2) edge (v3);
  \end{tikzpicture}
  =(\cdots)
  \quad\text{or}\quad 
  \begin{tikzpicture}
    \node[int] (i) at (0,.5) {};
    \node[int] (j) at (1,.5) {};
    \node (v1) at (-.5,-.2) {$\omega$};
    \node (v2) at (0,-.2) {$i$};
    \node (v3) at (1,-.2) {$\omega$};
    \node (v4) at (1.5,-.2) {$\omega$};
  \draw (i) edge[crossed] (v1) edge (v2) edge (j) (j) edge (v3) edge (v4);
  \end{tikzpicture}
  \quad\text{or}\quad 
  \begin{tikzpicture}
    \node[int] (i) at (0,.5) {};
    \node (v1) at (-.7,-.2) {$\omega$};
    \node (v2) at (-.2,-.2) {$\omega$};
    \node (v3) at (.2,-.2) {$\omega$};
    \node (v4) at (.7,-.2) {$i$};
  \draw (i) edge[crossed] (v1) edge (v2) edge (v3) edge (v4);
  \end{tikzpicture}
  \quad\text{or}\quad  
  \begin{tikzpicture}
    \node[int] (i) at (0,.5) {};
    \node[int] (j) at (1,.5) {};
    \node (v1) at (-.5,-.2) {$\omega$};
    \node (v2) at (0,-.2) {$\omega$};
    \node (v3) at (1,-.2) {$\omega$};
    \node (v4) at (1.5,-.2) {$i$};
  \draw (i) edge (v1) edge (v2) edge[crossed] (j) (j) edge (v3) edge (v4);
  \end{tikzpicture}
  \\&
  \quad\text{or}\quad 
  \begin{tikzpicture}
    \node[int] (i) at (0,.5) {};
    \node (v1) at (-.3,-.2) {$\omega$};
    \node (v2) at (.3,-.2) {$\omega$};
  \draw (i) edge[crossed, loop] (i) edge (v2) edge (v1);
  \end{tikzpicture}
  =0
  \quad\text{or}\quad
  \begin{tikzpicture}
    \node[int] (i) at (0,.5) {};
    \node[int] (j) at (0,1.2) {};
    \node (v1) at (-.3,-.2) {$\omega$};
    \node (v2) at (.3,-.2) {$\omega$};
  \draw (j) edge[crossed, loop] (j) edge (i) (i) edge (v2) edge (v1);
  \end{tikzpicture}=0
  \quad\text{or}\quad 
  \begin{tikzpicture}
    \node[int] (i) at (0,.5) {};
    \node[int] (j) at (1,.5) {};
    \node (v1) at (-.5,-.2) {$\omega$};
    \node (v2) at (0,-.2) {$\omega$};
    \node (v3) at (1,-.2) {$\omega$};
    \node (v4) at (1.5,-.2) {$\omega$};
    \node (v5) at (.5,-.2) {$\omega$};
  \draw (i) edge (v1) edge (v2) edge (v5) edge[crossed] (j) (j) edge (v3) edge (v4);
  \end{tikzpicture}
  \quad\text{or}\quad 
  \begin{tikzpicture}
    \node[int] (i) at (0,.5) {};
    \node[int] (j) at (1,.5) {};
    \node[int] (k) at (-1,.5) {};
    \node (v1) at (-1.5,-.2) {$\omega$};
    \node (v2) at (0,-.2) {$\omega$};
    \node (v3) at (1,-.2) {$\omega$};
    \node (v4) at (1.5,-.2) {$\omega$};
    \node (v5) at (-1,-.2) {$\omega$};
  \draw (i) edge (k) edge (v2)  edge[crossed] (j) (j) edge (v3) edge (v4)
  (k) edge (v5) edge (v1);
  \end{tikzpicture}
  \\&
  \quad\text{or}\quad 
  \begin{tikzpicture}
    \node[int] (i) at (0,.5) {};
    \node[int] (j) at (1,.5) {};
    \node (v1) at (-.5,-.2) {$\omega$};
    \node (v2) at (0,-.2) {$\omega$};
    \node (v3) at (1,-.2) {$\omega$};
    \node (v4) at (1.5,-.2) {$\omega$};
    \node (v5) at (.5,-.2) {$\omega$};
  \draw (i) edge[crossed] (v1) edge (v2) edge (j) (j) edge (v3) edge (v4) edge (v5);
  \end{tikzpicture}=(\cdots)
  \quad\text{or}\quad
  \\& 
  \begin{tikzpicture}
    \node[int] (i) at (0,.5) {};
    \node (v1) at (-.7,-.2) {$\omega$};
    \node (v2) at (-.2,-.2) {$\omega$};
    \node (v3) at (.2,-.2) {$\omega$};
    \node (v4) at (.7,-.2) {$\omega$};
    \node (v5) at (1.2,-.2) {$\omega$};
  \draw (i) edge[crossed] (v1) edge (v2) edge (v3) edge (v4) edge (v5);
  \end{tikzpicture}
  \quad\text{or}\quad 
  \begin{tikzpicture}
    \node[int] (i) at (0,.5) {};
    \node[int] (j) at (1,.5) {};
    \node (v1) at (-.5,-.2) {$\omega$};
    \node (v2) at (0,-.2) {$\omega$};
    \node (v3) at (1,-.2) {$\omega$};
    \node (v4) at (1.5,-.2) {$\omega$};
    \node (v5) at (.5,-.2) {$\omega$};
  \draw (i) edge[crossed] (v1) edge (v2) edge (v5) edge (j) (j) edge (v3) edge (v4);
  \end{tikzpicture}
  \\&
  \begin{tikzpicture}
    \node[int] (i) at (0,.5) {};
    \node[int] (j) at (1,.5) {};
    \node[int] (k) at (-1,.5) {};
    \node (v1) at (-1.5,-.2) {$\omega$};
    \node (v2) at (0,-.2) {$\omega$};
    \node (v3) at (1,-.2) {$\omega$};
    \node (v4) at (1.5,-.2) {$\omega$};
    \node (v5) at (-1,-.2) {$\omega$};
  \draw (i) edge (k) edge (v2)  edge (j) (j) edge (v3) edge (v4)
  (k) edge (v5) edge[crossed] (v1);
  \end{tikzpicture}=(\cdots)
  \quad\text{or}\quad
  \begin{tikzpicture}
    \node[int] (i) at (0,.5) {};
    \node[int] (j) at (1,.5) {};
    \node[int] (k) at (-1,.5) {};
    \node (v1) at (-1.5,-.2) {$\omega$};
    \node (v2) at (0,-.2) {$\omega$};
    \node (v3) at (1,-.2) {$\omega$};
    \node (v4) at (1.5,-.2) {$\omega$};
    \node (v5) at (-1,-.2) {$\omega$};
  \draw (i) edge (k) edge[crossed] (v2)  edge (j) (j) edge (v3) edge (v4)
  (k) edge (v5) edge (v1);
  \end{tikzpicture}
  \substack{\text{symm.} \\ =}0
\end{align*}
The first two graphs in the third line are zero in $\bGK_{g,n}^{12,1}$ because of rule (3) above. Additionally, some graphs can be expressed through other generators via the relations, which is written as $=(\cdots)$, see again Remark \ref{rem:graph vanishing}. 
We hence end up with an additional 8 generators with one blown-up component of excess 2.
Collecting all 20 generators, the differential is as follows:

\smallskip

\underline{Degree $25-n$:}
\begin{align*}
\Gamma_{i}^{(2)}:=
\scalebox{.6}{$\omis\tripods  
  \begin{tikzpicture}
    \node[int] (i) at (0,.5) {};
    \node[int] (j) at (1,.5) {};
    \node (v1) at (-.5,-.2) {$\omega$};
    \node (v2) at (0,-.2) {$i$};
    \node (v3) at (1,-.2) {$\omega$};
    \node (v4) at (1.5,-.2) {$\omega$};
  \draw (i) edge[crossed] (v1) edge (v2) edge (j) (j) edge (v3) edge (v4);
  \end{tikzpicture} $}
  &\mapsto 0
  \\
  \Gamma_{ij}^{(2)}:=
\scalebox{.6}{$\omis
\tripods
\begin{tikzpicture}
    \node[int] (i) at (0,.5) {};
    \node (v1) at (-.5,-.2) {$\omega$};
    \node (v2) at (0,-.2) {$i$};
    \node (v3) at (.5,-.2) {$j$};
  \draw (i) edge[crossed] (v1) edge (v2) edge (v3);
  \end{tikzpicture}
  $}
  &\mapsto 0
  \\
\Gamma^{(2)}_{j,irr}:=
\scalebox{.6}{$
  \begin{tikzpicture}
    \node (v) at (0,0) {$\omega$};
    \node (w) at (1,0) {$1$};
    \draw (v) edge (w);
  \end{tikzpicture}
  \cdots 
  \begin{tikzpicture}
    \node (v) at (0,0) {$\omega$};
    \node (w) at (1,0) {$n$};
    \draw (v) edge (w);
  \end{tikzpicture}
  \begin{tikzpicture}
    \node[int] (i) at (0,.5) {};
    \node (v1) at (-.5,-.2) {$\omega$};
    \node (v2) at (0,-.2) {$\omega$};
    \node (v3) at (.5,-.2) {$\omega$};
  \draw (i) edge (v1) edge (v2) edge (v3);
  \end{tikzpicture}
  \cdots 
  \begin{tikzpicture}
    \node[int] (i) at (0,.5) {};
    \node (v1) at (-.5,-.2) {$\omega$};
    \node (v2) at (0,-.2) {$\omega$};
    \node (v3) at (.5,-.2) {$\omega$};
  \draw (i) edge (v1) edge (v2) edge (v3);
  \end{tikzpicture}
  \begin{tikzpicture}
    \node[int] (i) at (0,.5) {};
    \node (v1) at (-.5,-.2) {$j$};
    \node (v2) at (0,-.2) {$\omega$};
    \node (v3) at (.5,-.2) {$\omega$};
  \draw (i) edge (v1) edge (v2) edge (v3);
  \end{tikzpicture}
  \begin{tikzpicture}
    \node[int] (i) at (0,.5) {};
    \node (v2) at (0,-.2) {$\omega$};
  \draw (i) edge[crossed, loop] (i) edge (v2) ;
  \end{tikzpicture}
  $}
  &\mapsto 0
\\
\Gamma_{birr}^{(2)}:=
\scalebox{.6}{$\omis\tripods 
\begin{tikzpicture}
  \node[int] (i) at (0,.5) {};
  \node[int] (j) at (1,.5) {};
  \node (v1) at (-.5,-.2) {$\omega$};
  \node (v2) at (0,-.2) {$\omega$};
  \node (v3) at (1,-.2) {$\omega$};
  \node (v4) at (1.5,-.2) {$\omega$};
\draw (i) edge (v1) edge (v2) edge (j) (j) edge (v3) edge (v4);
\end{tikzpicture}
\begin{tikzpicture}
        \node[int] (i) at (0,.5) {};
        \node (v2) at (0,-.2) {$\omega$};
      \draw (i) edge[crossed, loop] (i) edge (v2) ;
\end{tikzpicture}
$}&\mapsto 0
\\
\Gamma_{b\bar b}^{(2)}:=
\scalebox{.6}{$\omis\tripods 
\begin{tikzpicture}
  \node[int] (i) at (0,.5) {};
  \node[int] (j) at (1,.5) {};
  \node (v1) at (-.5,-.2) {$\omega$};
  \node (v2) at (0,-.2) {$\omega$};
  \node (v3) at (1,-.2) {$\omega$};
  \node (v4) at (1.5,-.2) {$\omega$};
\draw (i) edge (v1) edge (v2) edge (j) (j) edge (v3) edge (v4);
\end{tikzpicture}
\begin{tikzpicture}
        \node[int] (i) at (0,.5) {};
        \node[int] (j) at (1,.5) {};
        \node (v1) at (-.5,-.2) {$\omega$};
        \node (v2) at (0,-.2) {$\omega$};
        \node (v3) at (1,-.2) {$\omega$};
        \node (v4) at (1.5,-.2) {$\omega$};
      \draw (i) edge (v1) edge (v2) edge[crossed] (j) (j) edge (v3) edge (v4);
      \end{tikzpicture}
$}
&\mapsto 0
\\
\Gamma_{j\bar b}^{(2)}:=
\scalebox{.6}{$\omis\tripods 
\begin{tikzpicture}
  \node[int] (i) at (0,.5) {};
  \node (v1) at (-.5,-.2) {$j$};
  \node (v2) at (0,-.2) {$\omega$};
  \node (v3) at (.5,-.2) {$\omega$};
\draw (i) edge (v1) edge (v2) edge (v3);
\end{tikzpicture}
  \begin{tikzpicture}
    \node[int] (i) at (0,.5) {};
    \node[int] (j) at (1,.5) {};
    \node (v1) at (-.5,-.2) {$\omega$};
    \node (v2) at (0,-.2) {$\omega$};
    \node (v3) at (1,-.2) {$\omega$};
    \node (v4) at (1.5,-.2) {$\omega$};
  \draw (i) edge (v1) edge (v2) edge[crossed] (j) (j) edge (v3) edge (v4);
  \end{tikzpicture}
$}
&\mapsto 0
\\
  \Gamma_{\bar bi}^{(2)}:=
\scalebox{.6}{$\omis\tripods 
  \begin{tikzpicture}
    \node[int] (i) at (0,.5) {};
    \node[int] (j) at (1,.5) {};
    \node (v1) at (-.5,-.2) {$\omega$};
    \node (v2) at (0,-.2) {$\omega$};
    \node (v3) at (1,-.2) {$\omega$};
    \node (v4) at (1.5,-.2) {$i$};
  \draw (i) edge (v1) edge (v2) edge[crossed] (j) (j) edge (v3) edge (v4);
  \end{tikzpicture}$}
  &\mapsto 0
  \\
  \Gamma_{\bar b\bar b}^{(2)}:= \scalebox{.6}{$\omis\tripods  
  \begin{tikzpicture}
    \node[int] (i) at (0,.5) {};
    \node[int] (j) at (1,.5) {};
    \node[int] (k) at (-1,.5) {};
    \node (v1) at (-1.5,-.2) {$\omega$};
    \node (v2) at (0,-.2) {$\omega$};
    \node (v3) at (1,-.2) {$\omega$};
    \node (v4) at (1.5,-.2) {$\omega$};
    \node (v5) at (-1,-.2) {$\omega$};
  \draw (i) edge (k) edge (v2)  edge[crossed] (j) (j) edge (v3) edge (v4)
  (k) edge (v5) edge (v1);
  \end{tikzpicture}$}
  &\mapsto 0
\end{align*}

Note that the generators $\Gamma_{ij}^{(2)}$, $\Gamma_{i}^{(2)}$ are not independent; there are relations from the weight 13 relations, depending on the genus.
If $g=1$, only $\Gamma_{ij}^{(2)}$ contributes, and modulo relations the contribution is the irreducible representation $V_{21^{n-2}}$.
If $g\geq 3$ and $n\geq 2$, the $\Gamma^{(2)}_{i}$ are linearly independent, the $\Gamma_{ij}^{(2)}$ can be expressed through the $\Gamma^{(2)}_{i}$, and the overall $\ss_n$-representation is $V_{1^n} \oplus V_{21^{n-2}}$.

\smallskip

\underline{Degree $24-n$:}
\begin{align*}
\Gamma_{4irr}^{(2)}:= \scalebox{.6}{$\omis\tripods 
\begin{tikzpicture}
  \node[int] (i) at (0,.5) {};
  \node (v1) at (-.6,-.2) {$\omega$};
  \node (v2) at (-.2,-.2) {$\omega$};
  \node (v3) at (.2,-.2) {$\omega$};
  \node (v4) at (.6,-.2) {$\omega$};
\draw (i) edge (v1) edge (v2) edge (v3) edge (v4);
\end{tikzpicture}
\begin{tikzpicture}
        \node[int] (i) at (0,.5) {};
        \node (v2) at (0,-.2) {$\omega$};
      \draw (i) edge[crossed, loop] (i) edge (v2) ;
\end{tikzpicture}
$}
&\mapsto 3
\Gamma_{birr}^{(2)}
\\
\Gamma_{4\bar b}^{(2)}:=\scalebox{.6}{$\omis\tripods 
\begin{tikzpicture}
  \node[int] (i) at (0,.5) {};
  \node (v1) at (-.6,-.2) {$\omega$};
  \node (v2) at (-.2,-.2) {$\omega$};
  \node (v3) at (.2,-.2) {$\omega$};
  \node (v4) at (.6,-.2) {$\omega$};
\draw (i) edge (v1) edge (v2) edge (v3) edge (v4);
\end{tikzpicture}
\begin{tikzpicture}
        \node[int] (i) at (0,.5) {};
        \node[int] (j) at (1,.5) {};
        \node (v1) at (-.5,-.2) {$\omega$};
        \node (v2) at (0,-.2) {$\omega$};
        \node (v3) at (1,-.2) {$\omega$};
        \node (v4) at (1.5,-.2) {$\omega$};
      \draw (i) edge (v1) edge (v2) edge[crossed] (j) (j) edge (v3) edge (v4);
      \end{tikzpicture}
$}
&\mapsto 3
\Gamma_{b\bar b}^{(2)}
\\
\Gamma_{j\bar 4}^{(2)}:=\scalebox{.6}{$\omis\tripods 
\begin{tikzpicture}
  \node[int] (i) at (0,.5) {};
  \node (v1) at (-.5,-.2) {$j$};
  \node (v2) at (0,-.2) {$\omega$};
  \node (v3) at (.5,-.2) {$\omega$};
\draw (i) edge (v1) edge (v2) edge (v3);
\end{tikzpicture}
  \begin{tikzpicture}
    \node[int] (i) at (0,.5) {};
    \node (v1) at (-.7,-.2) {$\omega$};
    \node (v2) at (-.2,-.2) {$\omega$};
    \node (v3) at (.2,-.2) {$\omega$};
    \node (v4) at (.7,-.2) {$\omega$};
  \draw (i) edge[crossed] (v1) edge (v2) edge (v3) edge (v4);
  \end{tikzpicture}
$}
&\mapsto -
\Gamma_{j\bar b}^{(2)}+ (const) \Gamma_{j,irr}^{(2)}
\\
\Gamma^{(2)}_{\epsilon,irr}:=
\scalebox{.6}{$
  \begin{tikzpicture}
    \node (v) at (0,0) {$\omega$};
    \node (w) at (1,0) {$1$};
    \draw (v) edge (w);
  \end{tikzpicture}
  \cdots 
  \begin{tikzpicture}
    \node (v) at (0,0) {$\omega$};
    \node (w) at (1,0) {$n$};
    \draw (v) edge (w);
  \end{tikzpicture}
  \begin{tikzpicture}
    \node[int] (i) at (0,.5) {};
    \node (v1) at (-.5,-.2) {$\omega$};
    \node (v2) at (0,-.2) {$\omega$};
    \node (v3) at (.5,-.2) {$\omega$};
  \draw (i) edge (v1) edge (v2) edge (v3);
  \end{tikzpicture}
  \cdots 
  \begin{tikzpicture}
    \node[int] (i) at (0,.5) {};
    \node (v1) at (-.5,-.2) {$\omega$};
    \node (v2) at (0,-.2) {$\omega$};
    \node (v3) at (.5,-.2) {$\omega$};
  \draw (i) edge (v1) edge (v2) edge (v3);
  \end{tikzpicture}
  \begin{tikzpicture}
    \node (v) at (0,0) {$\omega$};
    \node (w) at (1,0) {$\epsilon$};
    \draw (v) edge (w);
  \end{tikzpicture}
  \begin{tikzpicture}
    \node[int] (i) at (0,.5) {};
    \node (v2) at (0,-.2) {$\omega$};
  \draw (i) edge[crossed, loop] (i) edge (v2) ;
  \end{tikzpicture}
  $}
  &\mapsto  \sum_j \pm \Gamma_{j,irr}^{(2)} +(const) \Gamma_{birr}^{(2)}
\\
\Gamma_{\bar B}^{(2)}:=
\scalebox{.6}{$\omis
\tripods
\begin{tikzpicture}
    \node[int] (i) at (0,.5) {};
    \node[int] (j) at (1,.5) {};
    \node (v1) at (-.5,-.2) {$\omega$};
    \node (v2) at (0,-.2) {$\omega$};
    \node (v3) at (1,-.2) {$\omega$};
    \node (v4) at (1.5,-.2) {$\omega$};
    \node (v5) at (.5,-.2) {$\omega$};
  \draw (i) edge (v1) edge (v2) edge (v5) edge[crossed] (j) (j) edge (v3) edge (v4);
  \end{tikzpicture}
  $}
  &\mapsto 2
\Gamma_{\bar b\bar b}^{(2)}
  \\
  \Gamma_{\bar 4i}^{(2)}:=
  \scalebox{.6}{$\omis\tripods  
      \begin{tikzpicture}
    \node[int] (i) at (0,.5) {};
    \node (v1) at (-.7,-.2) {$\omega$};
    \node (v2) at (-.2,-.2) {$\omega$};
    \node (v3) at (.2,-.2) {$\omega$};
    \node (v4) at (.7,-.2) {$i$};
  \draw (i) edge[crossed] (v1) edge (v2) edge (v3) edge (v4);
  \end{tikzpicture} $}
  &\mapsto 
  - \Gamma_{\bar bi}^{(2)} + (const) \Gamma_{i irr}^{(2)}
  \\
  \Gamma_{\bar B'}^{(2)}:=
  \scalebox{.6}{$\omis\tripods  
    \begin{tikzpicture}
    \node[int] (i) at (0,.5) {};
    \node[int] (j) at (1,.5) {};
    \node (v1) at (-.5,-.2) {$\omega$};
    \node (v2) at (0,-.2) {$\omega$};
    \node (v3) at (1,-.2) {$\omega$};
    \node (v4) at (1.5,-.2) {$\omega$};
    \node (v5) at (.5,-.2) {$\omega$};
  \draw (i) edge[crossed] (v1) edge (v2) edge (v5) edge (j) (j) edge (v3) edge (v4);
  \end{tikzpicture} $}
  &\mapsto
  \pm \Gamma_{\bar b\bar b}^{(2)}+(const) \Gamma_{birr}^{(2)}
  \\
  \Gamma_{b\bar 4}^{(2)}:=
  \scalebox{.6}{$\omis\tripods 
\begin{tikzpicture}
  \node[int] (i) at (0,.5) {};
  \node[int] (j) at (1,.5) {};
  \node (v1) at (-.5,-.2) {$\omega$};
  \node (v2) at (0,-.2) {$\omega$};
  \node (v3) at (1,-.2) {$\omega$};
  \node (v4) at (1.5,-.2) {$\omega$};
\draw (i) edge (v1) edge (v2) edge (j) (j) edge (v3) edge (v4);
\end{tikzpicture}
  \begin{tikzpicture}
    \node[int] (i) at (0,.5) {};
    \node (v1) at (-.7,-.2) {$\omega$};
    \node (v2) at (-.2,-.2) {$\omega$};
    \node (v3) at (.2,-.2) {$\omega$};
    \node (v4) at (.7,-.2) {$\omega$};
  \draw (i) edge[crossed] (v1) edge (v2) edge (v3) edge (v4);
  \end{tikzpicture}
$}
&\mapsto - \Gamma_{b\bar b}^{(2)} + (const) \Gamma_{birr}^{(2)}
\\
\Gamma_{\epsilon\bar b}^{(2)}
:=
\scalebox{.6}{$\omis\tripods 
  \begin{tikzpicture}
    \node (v) at (0,0) {$\omega$};
    \node (w) at (1,0) {$\epsilon$};
    \draw (v) edge (w);
  \end{tikzpicture}
  \begin{tikzpicture}
    \node[int] (i) at (0,.5) {};
    \node[int] (j) at (1,.5) {};
    \node (v1) at (-.5,-.2) {$\omega$};
    \node (v2) at (0,-.2) {$\omega$};
    \node (v3) at (1,-.2) {$\omega$};
    \node (v4) at (1.5,-.2) {$\omega$};
  \draw (i) edge (v1) edge (v2) edge[crossed] (j) (j) edge (v3) edge (v4);
  \end{tikzpicture}
$}
&\mapsto \sum_j \pm \Gamma_{j\bar b}^{(2)} + (const) \Gamma_{b\bar b}^{(2)}
\pm 4 \Gamma_{\bar b\bar b}^{(2)}
\end{align*}

\smallskip

\underline{Degree $23-n$:}
\[
\resizebox{.95\hsize}{!}{
$\begin{aligned}
\Gamma_{4\bar 4}^{(2)}:=\scalebox{.6}{$\omis\tripods 
\begin{tikzpicture}
  \node[int] (i) at (0,.5) {};
  \node (v1) at (-.6,-.2) {$\omega$};
  \node (v2) at (-.2,-.2) {$\omega$};
  \node (v3) at (.2,-.2) {$\omega$};
  \node (v4) at (.6,-.2) {$\omega$};
\draw (i) edge (v1) edge (v2) edge (v3) edge (v4);
\end{tikzpicture}
  \begin{tikzpicture}
    \node[int] (i) at (0,.5) {};
    \node (v1) at (-.7,-.2) {$\omega$};
    \node (v2) at (-.2,-.2) {$\omega$};
    \node (v3) at (.2,-.2) {$\omega$};
    \node (v4) at (.7,-.2) {$\omega$};
  \draw (i) edge[crossed] (v1) edge (v2) edge (v3) edge (v4);
  \end{tikzpicture}
$}
&\mapsto 3 \Gamma_{b\bar 4}^{(2)} + (const) \Gamma_{4b}^{(2)} + (const) \Gamma_{4irr}^{(2)}
\\
\Gamma^{(2)}_{\epsilon\bar 4}:=
\scalebox{.6}{$\omis\tripods 
  \begin{tikzpicture}
    \node (v) at (0,0) {$\omega$};
    \node (w) at (1,0) {$\epsilon$};
    \draw (v) edge (w);
  \end{tikzpicture}
  \begin{tikzpicture}
    \node[int] (i) at (0,.5) {};
    \node (v1) at (-.7,-.2) {$\omega$};
    \node (v2) at (-.2,-.2) {$\omega$};
    \node (v3) at (.2,-.2) {$\omega$};
    \node (v4) at (.7,-.2) {$\omega$};
  \draw (i) edge[crossed] (v1) edge (v2) edge (v3) edge (v4);
  \end{tikzpicture}
$}
&\mapsto - \Gamma_{\epsilon \bar b}^{(2)} + \sum_j \pm  \Gamma_{j\bar 4}^{(2)} 
+ (const) \Gamma_{b\bar 4}^{(2)} 
\pm \Gamma_{\bar B}^{(2)} \pm 3 \Gamma_{\bar B'}^{(2)}
\\
\Gamma_{\bar 5}^{(2)} := \scalebox{.6}{$\omis\tripods  
\begin{tikzpicture}
    \node[int] (i) at (0,.5) {};
    \node (v1) at (-.7,-.2) {$\omega$};
    \node (v2) at (-.2,-.2) {$\omega$};
    \node (v3) at (.2,-.2) {$\omega$};
    \node (v4) at (.7,-.2) {$\omega$};
    \node (v5) at (1.2,-.2) {$\omega$};
  \draw (i) edge[crossed] (v1) edge (v2) edge (v3) edge (v4) edge (v5);
  \end{tikzpicture} $}
  &\mapsto 6 \Gamma_{\bar B'}^{(2)} + (const) \Gamma_{4irr}^{(2)} + (const) \Gamma_{\bar B}^{(2)}
\end{aligned}$}
\]

We can now compute the cohomology of $\myB_{g,n}$ for $3g+2n=27$. The complex is concentrated in degrees $23-n$, $24-n$, $25-n$. Note that some generators exist only for higher genus or higher $n$. There is no cocycle of degree $23-n$, since the image of the differential is of full dimension. This is seen already by looking only at the leading terms $\Gamma_{b\bar 4}^{(2)}$, $\Gamma_{\epsilon \bar b}^{(2)}$, $\Gamma_{\bar B'}^{(2)}$.
Note that these generators only exist for $g$ large enough, i.e., $g\geq 4$ or $g\geq 7$ respectively. But if they do not (because $g$ is too small), the same holds for the corresponding generators of degree $23-n$.
Hence we conclude that $H^{23-n}(\myB_{g,n})=0$ for all $g,n$ considered.

Next we consider cocycles $x$ of degree $24-n$.
The general cocycle is a linear combination of the $9$ generators in degree $24-n$. By adding an exact term to $x$ we may however assume that this linear combination does not involve $\Gamma_{b\bar 4}^{(2)}$, $\Gamma_{\epsilon \bar b}^{(2)}$, $\Gamma_{\bar B'}^{(2)}$.
Assume first that $n\geq 1$. Then we claim that no linear combination of the remaining generators can be closed. This is so because the image of the generators under the differential is already of full rank if projected to the subspace spanned by the ``leading terms" $\Gamma_{birr}^{(2)}$, $\Gamma_{b\bar b}^{(2)}$,
$\Gamma_{j\bar b}^{(2)}$,
$\Gamma_{j,irr}^{(2)}$,
$\Gamma_{\bar b\bar b}^{(2)}$,
$\Gamma_{\bar bi}^{(2)}$.
Hence we have $H^{24-n}(\myB_{g,n})=0$ for $n\geq 1$.
In the special case $n=0$, i.e., $g=9$, the generator $\Gamma_{j,irr}^{(2)}$ of degree $25$ does not exist, while $\Gamma_{\epsilon,irr}^{(2)}$ does exist.
Hence a linear combination of $\Gamma_{\epsilon,irr}^{(2)}$ and $\Gamma_{4irr}^{(2)}$ is a cocycle, so that $H^{24}(\myB_{9,0})$ is one-dimensional.

Finally, we consider degree $25-n$. 
Any element $x$ of that degree is a cocyle. 
By adding exact terms we may ensure that $x$ does not involve the generators 
$\Gamma_{birr}^{(2)}$,
$\Gamma_{b\bar b}^{(2)}$,
$\Gamma_{j\bar b}^{(2)}$,
$\Gamma_{\bar b\bar b}^{(2)}$,
$\Gamma_{\bar bi}^{(2)}$,
and in addition we have to mod out the linear combination $\sum_j\pm \Gamma_{j,irr}^{(2)}$ (for $g>1$, $n>0$), that accounts for one copy of the sign representation $V_{1^n}$ of $\ss_n$.
Our $x$ can hence be a linear combination of the generators $\Gamma_i^{(2)}$,
$\Gamma_{j,irr}^{(2)}$ and $\Gamma_{ij}^{(2)}$.
The generators $\Gamma_i^{(2)}$ exist for $g\geq 3$ and $n\geq 1$ and contribute the representation $V_{1^n} \oplus V_{21^{n-2}}$ of $\bbS_n$.
As explained above the generators $\Gamma_{j,irr}^{(2)}$ and $\Gamma_{ij}^{(2)}$ in total contribute an $\ss_n$ representation $V_{21^{n-2}}$ if $g=1$, and a representation $V_{1^n}\oplus V_{21^{n-2}}$ for $g>1$, $n>0$. As mentioned above, in the case $g>1$ we have to remove the $V_{1^n}$ again (since it is in the image of the differential).
Hence we arrive at the following cohomology table:

\begin{align*}
  H^k(\myB_{1,12}) &=
  \begin{cases}
      V_{21^{10}} & \text{for $k=13$} \\
      0 & \text{otherwise}
  \end{cases}
  &
  H^k(\myB_{3,9}) &=
  \scalebox{.95}{$\begin{cases}
      V_{1^{9}}\oplus V_{21^7}\oplus V_{21^7} & \text{for $k=16$} \\
      0 & \text{otherwise}
  \end{cases}$
  }
  \\
  H^k(\myB_{5,6}) &=
  \begin{cases}
    V_{1^{6}}\oplus V_{21^4}\oplus V_{21^4} & \text{for $k=19$} \!\! \\
    0 & \text{otherwise}
\end{cases}
  &
  H^k(\myB_{7,3}) &=
  \begin{cases}
    V_{1^{3}}\oplus V_{21}\oplus V_{21} & \text{for $k=22$} \\
    0 & \text{otherwise}
\end{cases}
  \\
  H^k(\myB_{9,0}) &=
  \begin{cases}
    \mathbb C & \text{for $k=24$} \\
    0 & \text{otherwise.}
  \end{cases}
 \end{align*}

\noindent This completes the proof of Theorem~\ref{thm:lowexc13}.

\section{The \texorpdfstring{$\mathsf{S}_{16}$}{S16}-isotypic part of \texorpdfstring{$H^{15}(\Mb_{g,n})$}{H15}}

In this section, we prove Theorem \ref{thm:H15}, describing the $\mathsf{S}_{16}$-isotypic part of $H^{15}(\Mb_{g,n})$. By \cite[Theorem 1.1]{CLP-STE}, the $\mathsf{S}_{16}$-isotypic part of $H^{15}(\Mb_{g,n})$ vanishes for $g\geq 2$. Moreover, $H^{15}(\Mb_{0,n})=0$ for all $n$. It remains to consider the case $g=1$.
\begin{proof}[Proof of Theorem \ref{thm:H15}]
By the argument in \cite[Section~4.1]{CLPW}, it suffices to work in the category of Hodge structures. We have a right exact sequence of Hodge structures
\[ H^{13}(\tilde{\partial \M_{1,n}}) \xrightarrow{\alpha} H^{15}(\Mb_{1,n}) \rightarrow W_{15}H^{15}(\M_{1,n})\rightarrow 0,
\]
where the first morphism $\alpha$ is of Hodge type $(1,1)$. The image of $\alpha$ is therefore disjoint from $H^{15,0}(\Mb_{1,n})\oplus H^{0,15}(\Mb_{1,n})$. Because $\mathsf{S}_{16}\otimes \mathbb{C}$ is of Hodge type $(15,0), (0,15)$, it follows that the $\mathsf{S}_{16}$-isotypic parts of $W_{15}H^{15}(\M_{1,n})$ and $H^{15}(\Mb_{1,n})$ agree. By \cite[Proposition 2.2]{CLP-STE}, 
$W_{15}H^{15}(\M_{1,n})\cong K^{15}_n\otimes \mathsf{S}_{16}$.
\end{proof}
\begin{rem}
    Essentially the same argument shows that the $\mathsf{S}_{k+1}$-isotypic part of $H^{k}(\Mb_{1,n})$ is $K^k_n \otimes \s_{k+1}$. For odd $k \geq 17$ and $g\geq 2$, the $\mathsf{S}_{k+1}$-isotypic part of $H^k(\Mb_{g,n})$ is not well understood. For example, it is known that the $\mathsf{S}_{18}$-isotypic part of $H^{17}(\Mb_{2,14})$ is nontrivial in the category of $\ell$-adic Galois representations, by \cite{Petersenlocalsystems}. On the automorphic side of the Langlands correspondence, this corresponds to a Saito--Kurokawa lift of the cusp form $\omega$ of level 1 and weight 18 from $\mathrm{SL}_2(\ZZ)$ to $\Sp_4(\ZZ)$, via relations between the cohomology of $\M_{2,n}$ and that of local systems on $\mathcal{A}_2$.
    This cusp form $\omega$ can also be used to construct an explicit holomorphic $17$-form on $\Mb_{2,14}$ \cite[Section 3.5]{FPhandbook}, and hence the real Hodge structure $\mathsf{S}_{18} \otimes \RR$ appears in $H^{17}(\MM_{2,14})$. 
    However, it is an open problem to show that the rational Hodge structure $\mathsf{S}_{18}$ appears in the cohomology of $\MM_{2,14}$. Moreover, it is unknown (in both Hodge structures and $\ell$-adic Galois representations) whether the $\mathsf{S}_{18}$-isotypic part of $H^{17}(\Mb_{g,n})$ is non-zero in any case where $g\geq 3$.
\end{rem}

For subsequent calculations, we write an explicit basis for $H^{15,0}(\Mb_{1,n})$. A weight $16$ cusp form $\omega_{15}$ for $\SL_2(\zz)$ is a generator for $H^{15,0}(\Mb_{1,15})$, on which $\ss_{15}$ acts by the sign representation \cite{BergstromData}. Let $A\subset \{1,\dots,n\}$ be an ordered subset of size $15$. For each such $A$, we have a forgetful map
\[
f_{A}\colon \Mb_{1,n}\rightarrow \Mb_{1,A},
\]
and we define $\omega_{A}:=f_{A}^*\omega_{15}$. Following exactly the same argument as in \cite[Section 2.3]{CanningLarsonPayne}, one can show that the set of forms
$\{\omega_{A}: 1\in A, A \text{ increasing}\}$ forms a basis for $H^{15,0}(\Mb_{1,n})$. To each such subset $A$, we associate the tableau $T_A$ of shape $(n-14,1^{14})$ with symbols of $A$ in order down the first column and the rest of the first row filled in increasing order. The isomorphism $H^{15,0}(\Mb_{1,n})\rightarrow V_{(n-14)1^{14}}= K_n^{15}$ sends $\omega_{A}$ to the Specht module generator associated to $T_A$.

\section{Graph complexes for \texorpdfstring{$\M_{g,n}$}{Mgn} in type (15,0)} \label{sec:weight15}

In this section, we briefly describe 
how, in a manner similar to \cite{PayneWillwacher24},
one can construct complexes $B_{g,n}^{15}$ and $C_{g,n}^{15}$ whose cohomology computes the type $(15,0)$ part of $H^{*}_c(\M_{g,n})$.
The only difference between the complexes we describe here and the complexes in \cite{PayneWillwacher24}
is that the truncation happens at 15 marked half-edges.

\subsection{Graph complexes }
As discussed in Section~\ref{pics}, above, the cohomology operad $H(\MM)$ carries a pure Hodge structure that induces a decomposition of the Feynman transform
\[
\GK_{g,n}^k \otimes \cc \cong \bigoplus_{p+q = k} \GK_{g,n}^{p,q}.
\]
The type $(15,0)$ part $\GK^{15,0}_{g,n}$ computes the corresponding part of the weight-graded compactly supported cohomology of $\M_{g,n}$
\[
H^*(\GK^{15,0}_{g,n}) \cong \gr_{15,0} H^*_c(\M_{g,n}):= \gr_F^{15}\gr^W_{15} H^*_c(\M_{g,n},\cc).
\]
Here, $F$ denotes the Hodge filtration in the mixed Hodge structure. In other words, the graph complex $\GK^{15,0}_{g,n}$ computes the type $(15,0)$ part of the pure Hodge structure $\gr_{15}^W H^*_c(\M_{g,n})$.
The generators of $\GK^{15,0}_{g,n}$ are stable graphs with $n$ external legs, with one special vertex $v$ of genus 1 decorated by an element of $H^{15,0}(\MM_{1,n_v})$. All other vertices $w$ are decorated by an element of $H^0(\MM_{g_w,n_w})$, and since this vector space is one-dimensional, we will tacitly assume that they are decorated by 1 and omit the decoration from further discussion.
Recall the explicit description of $H^{15,0}(\MM_{1,n_v})$ of Theorem \ref{thm:H15}.
Graphically, we may denote the decoration $\omega_A\in H^{15,0}(\MM_{1,n_v})$ by marking the 15 half-edges in $A$ by arrows.
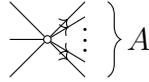
\[
\begin{tikzpicture}
    \node[ext] (v) at (0,0) {};
    \node at (.5, .1) {$\scriptscriptstyle \vdots$};
    \draw (v) edge +(-.5,-.5) edge +(-.5,0) edge +(-.5,+.5) 
    edge[->-] +(.5,+.5) edge[->-] +(.5,+.3) edge[->-] +(.5,-.3) edge[->-] +(.5,-.5)
    ;
    \draw [decorate,decoration={brace,amplitude=5pt}]
  (.8,.5) -- (.8,-.5) node[midway,xshift=1em]{$A$};
\end{tikzpicture}
\]
Furthermore, we define the graded subspace
\[
I_{g,n}^{15,0}\subset \GK^{15,0}_{g,n}
\]
spanned by graphs that have either a self-loop at a weight 0 vertex or a weight 0 vertex of positive genus. The subspace $I_{g,n}^{15,0}$ is evidently closed under the differential, $dI_{g,n}^{15,0}\subset I_{g,n}^{15,0}$, so that we may consider the quotient complex 
\[
\bGK^{15,0}_{g,n} := \GK^{15,0}_{g,n}/ I_{g,n}^{15,0}.
\]
\begin{lem}
    The projection to the quotient $\GK^{15,0}_{g,n}\to \bGK^{15,0}_{g,n}$ is a quasi-isomorphism.
\end{lem}
\begin{proof}
    The proof is identical to that of  \cite[Proposition 3.2]{PayneWillwacher24}. 
\end{proof}

To go further, recall from \cite[Section 3.1]{PayneWillwacher24} the definition of the auxiliary graph complex $X_{g,n}$. 
Generators of $X_{g,n}$ are graphs with one special vertex with an arbitrary subset of marked incident half-edges.
The differential is given by splitting vertices, and removing a half-edge from the set of marked half-edges. The complex $X_{g,n}$ is acyclic, i.e., $H(X_{g,n})=0$, see \cite[Lemma 3.4]{PayneWillwacher24}.
We define the subcomplex 
$$
X_{g,n}^{\leq r} \subset X_{g,n}
$$
spanned by graphs such that the number of marked half-edges is $\leq r$.
We then define 
\begin{align}
C_{g,n}^{15} &:= X_{g,n}^{\leq 14}[-29] \\
B_{g,n}^{15} &:= \bigg(X_{g,n} / X_{g,n}^{\leq 14}\bigg)[-30].
\end{align}
By acyclicity of $X_{g,n}$ we have that 
\begin{equation}\label{equ:HB HC iso}
H(B_{g,n}^{15}) \cong H(C_{g,n}^{15}),
\end{equation}
and the isomorphism is the connecting homomorphism in the natural long exact cohomology sequence associated to the short exact sequence 
\begin{equation}\label{equ:B C short exact}
0\to X_{g,n}^{\leq 14} \to X_{g,n} \to X_{g,n}/X_{g,n}^{\leq 14}\to 0.
\end{equation}
There is a natural morphism of dg vector spaces 
\[
B_{g,n}^{15} \to \bGK_{g,n}^{15,0}
\]
defined on generators by sending a graph to itself. See \cite[Definition 3.5ff.]{PayneWillwacher24} for the analogous construction in weight 11, which differs only in the replacement of the truncation constant 14 by 10.

\begin{prop}
The morphism $B_{g,n}^{15}\to \bGK_{g,n}^{15,0}$ is a quasi-isomorphism. In particular, 
\[
\gr_{15,0} H_c^*(\M_{g,n}) \cong H^*(B_{g,n}^{15}) \cong H^*(C_{g,n}^{15}).
\]
\end{prop}
\begin{proof}
The proof is identical to the proof of \cite[Proposition 3.6]{PayneWillwacher24}.
\end{proof}

In what follows, we will again draw graphs in the ``blown-up" picture, removing the special vertex $v$ so that the half-edges incident at $v$ become external legs. The legs corresponding to marked half-edges will be marked by a symbol $\omega$, and the non-marked half-edges by a symbol $\epsilon$. See \cite[Section 3.2]{PayneWillwacher24} or Section~\ref{sec:blownup}, above.

\newcommand{\yB}{B^{15}}
\section{Low excess computations in type (15,0)}\label{sec:low excess15}

We define the excess of a generator $\Gamma$ of $X_{g,n}$ to be
\[
E(\Gamma) = 3(g-1) +2n - 2\# \omega,  
\] 
where $\# \omega$ is the number of $\omega$-legs of $\Gamma$.
We also define
\[
E'(g,n) := 3g +2n - 33.
\]
Any generator $\Gamma$ for $\yB_{g,n}$ has $\# \omega\geq 15$, and hence
\[
E(\Gamma) \leq E'(g,n). 
\]
Furthermore $E(\Gamma) \equiv E'(g,n) \mod 2$.

Write each generator $\Gamma$ for $\yB_{g,n}$ as a union of its blown-up components:
\[
\Gamma = C_1\cup\cdots \cup C_k.
\]
The excess is additive 
\[
E(\Gamma) = E(C_1\cup\cdots \cup C_k)= E(C_1)+ \cdots + E(C_k),
\]
if we continue to define the excess of components by \eqref{equ:exc comp}. 
As in \cite[Lemma 4.2]{PayneWillwacher24}, the excess of all generators of $\yB_{g,n}$ is non-negative.
Hence, we obtain the following result, stated as Proposition \ref{prop:wt 15 vanishing} in the introduction.
\begin{cor} \label{cor:Eneg15}
If $E'(g,n) <0$ then $
\gr_{15,0} H^*_c(\M_{g,n}) =0.$
\end{cor}

The low-excess computations of \cite[Section 4]{PayneWillwacher24} in weight 11 carry over almost unaltered to the type $(15,0)$ situation.
We hence only briefly recall the arguments and state the results, leaving details to loc. cit.

\subsection{Case \texorpdfstring{$E'(g,n) = 0$}{E'=0}}
When $E'(g,n) = 0$, the generators of $B_{g,n}$ have the following form:
\begin{equation}\label{equ:Gamma 0}
  \Gamma^{(0)}=
  \begin{tikzpicture}
    \node (v) at (0,0) {$\omega$};
    \node (w) at (1,0) {$1$};
    \draw (v) edge (w);
  \end{tikzpicture}
  \cdots 
  \begin{tikzpicture}
    \node (v) at (0,0) {$\omega$};
    \node (w) at (1,0) {$n$};
    \draw (v) edge (w);
  \end{tikzpicture}
  \begin{tikzpicture}
    \node[int] (i) at (0,.5) {};
    \node (v1) at (-.5,-.2) {$\omega$};
    \node (v2) at (0,-.2) {$\omega$};
    \node (v3) at (.5,-.2) {$\omega$};
  \draw (i) edge (v1) edge (v2) edge (v3);
  \end{tikzpicture}
  \cdots 
  \begin{tikzpicture}
    \node[int] (i) at (0,.5) {};
    \node (v1) at (-.5,-.2) {$\omega$};
    \node (v2) at (0,-.2) {$\omega$};
    \node (v3) at (.5,-.2) {$\omega$};
  \draw (i) edge (v1) edge (v2) edge (v3);
  \end{tikzpicture}
\end{equation}
Note that there are $n$ $(\omega-j)$-edges and $\frac {g-1}{2}$ tripods with three $\omega$-legs each.
The cohomological degree of such a generator is $k=15+\frac32 (g-1)$.

One hence arrives at the following list of cases in which $\gr_{15,0} H_c^k(\M_{g,n})$ is concentrated in a single degree $k$. The $\bbS_n$-action is by the sign representation in each case.
\begin{align*}   H^k(\yB_{1,15}) &=
    \begin{cases}
        V_{1^{15}} & \text{for $k=15$} \\
        0 & \text{otherwise}
    \end{cases}
&
H^k(\yB_{3,12}) &=
    \begin{cases}
        V_{1^{12}} & \text{for $k=18$} \\
        0 & \text{otherwise}
    \end{cases}
\\
H^k(\yB_{5,9}) &=
\begin{cases}
    V_{1^9} & \text{for $k=21$} \\
    0 & \text{otherwise}
\end{cases}     
&
H^k(\yB_{7,6}) &=
\begin{cases}
    V_{1^6} & \text{for $k=24$} \\
    0 & \text{otherwise}
\end{cases}    
\\
H^k(\yB_{9,3}) &=
    \begin{cases}
        V_{1^3} & \text{for $k=27$} \\
        0 & \text{otherwise}
    \end{cases}     
&
H^k(\yB_{11,0}) &=
\begin{cases}
    \Q & \text{for $k=30$} \\
    0 & \text{otherwise}
\end{cases}   
\end{align*}
\noindent Note, in particular, that the computation for $B^{15}_{11,0}$ proves the first statement in Theorem~\ref{thm:s16}.

\subsection{Case \texorpdfstring{$E'(g,n) = 1$}{E'=1}} \label{sec:exc1 15}
When $E'(g,n) = 1$, the relevant generators of $\yB_{g,n}$ are of the form 
\[
  \Gamma^{(1)}=
  \begin{tikzpicture}
    \node (v) at (0,0) {$\omega$};
    \node (w) at (1,0) {$1$};
    \draw (v) edge (w);
  \end{tikzpicture}
  \cdots 
  \begin{tikzpicture}
    \node (v) at (0,0) {$\omega$};
    \node (w) at (1,0) {$n$};
    \draw (v) edge (w);
  \end{tikzpicture}
  \begin{tikzpicture}
    \node[int] (i) at (0,.5) {};
    \node (v1) at (-.5,-.2) {$\omega$};
    \node (v2) at (0,-.2) {$\omega$};
    \node (v3) at (.5,-.2) {$\omega$};
  \draw (i) edge (v1) edge (v2) edge (v3);
  \end{tikzpicture}
  \cdots 
  \begin{tikzpicture}
    \node[int] (i) at (0,.5) {};
    \node (v1) at (-.5,-.2) {$\omega$};
    \node (v2) at (0,-.2) {$\omega$};
    \node (v3) at (.5,-.2) {$\omega$};
  \draw (i) edge (v1) edge (v2) edge (v3);
  \end{tikzpicture}
  \begin{tikzpicture}
    \node (v) at (0,0) {$\omega$};
    \node (w) at (1,0) {$\epsilon$};
    \draw (v) edge (w);
  \end{tikzpicture}
\]
or 
\[
  \Gamma^{(1)}_j=
  \begin{tikzpicture}
    \node (v) at (0,0) {$\omega$};
    \node (w) at (1,0) {$1$};
    \draw (v) edge (w);
  \end{tikzpicture}
  \cdots 
  \begin{tikzpicture}
    \node (v) at (0,0) {$\omega$};
    \node (w) at (1,0) {$n$};
    \draw (v) edge (w);
  \end{tikzpicture}
  \begin{tikzpicture}
    \node[int] (i) at (0,.5) {};
    \node (v1) at (-.5,-.2) {$\omega$};
    \node (v2) at (0,-.2) {$\omega$};
    \node (v3) at (.5,-.2) {$\omega$};
  \draw (i) edge (v1) edge (v2) edge (v3);
  \end{tikzpicture}
  \cdots 
  \begin{tikzpicture}
    \node[int] (i) at (0,.5) {};
    \node (v1) at (-.5,-.2) {$\omega$};
    \node (v2) at (0,-.2) {$\omega$};
    \node (v3) at (.5,-.2) {$\omega$};
  \draw (i) edge (v1) edge (v2) edge (v3);
  \end{tikzpicture}
  \begin{tikzpicture}
    \node[int] (i) at (0,.5) {};
    \node (v1) at (-.5,-.2) {$j$};
    \node (v2) at (0,-.2) {$\omega$};
    \node (v3) at (.5,-.2) {$\omega$};
  \draw (i) edge (v1) edge (v2) edge (v3);
  \end{tikzpicture},
\]
with the understanding that there is no $(\omega-j)$-edge in $\Gamma^{(1)}_j$.
The differential sends $\Gamma^{(1)}$ to the alternating sum of $\Gamma_j^{(1)}$.
Hence the cohomology of $\yB_{g,n}$ is the irreducible $\bbS_n$-representation $V_{21^{n-2}}$ in degree $k=14+\frac 32 g$, represented by linear combinations of the $\Gamma_j^{(1)}$.
We obtain:
\begin{align*}
    H^k(\yB_{2,14}) &=
    \begin{cases}
        V_{21^{12}} & \text{for $k=17$} \\
        0 & \text{otherwise}
    \end{cases}
&
    H^k(\yB_{4,11}) &=
    \begin{cases}
        V_{21^{9}} & \text{for $k=20$} \\
        0 & \text{otherwise}
    \end{cases}
    \\
    H^k(\yB_{6,8}) &=
    \begin{cases}
        V_{21^{6}} & \text{for $k=23$} \\
        0 & \text{otherwise}
    \end{cases}
&
    H^k(\yB_{8,5}) &=
    \begin{cases}
        V_{21^{3}} & \text{for $k=26$} \\
        0 & \text{otherwise}
    \end{cases}
\\
    H^k(\yB_{10,2}) &=
    \begin{cases}
        V_{2} & \text{for $k=29$} \\
        0 & \text{otherwise}
    \end{cases}.
\end{align*}

\subsection{Case \texorpdfstring{$E'(g,n)=2$}{E'=2}} \label{sec:exc2-15}
Suppose $E'(g,n) = 2$.  Realizing some cancellations one sees that the relevant generators for $\yB_{g,n}$ are $\Gamma^{(0)}$ as in \eqref{equ:Gamma 0} above of degree $15+\frac32(g-1)$, 
together with the following graphs of excess 2:
\[
    \Gamma^{(2)}_{\epsilon}:=
  \begin{tikzpicture}
    \node (v) at (0,0) {$\omega$};
    \node (w) at (1,0) {$1$};
    \draw (v) edge (w);
  \end{tikzpicture}
  \cdots 
  \begin{tikzpicture}
    \node (v) at (0,0) {$\omega$};
    \node (w) at (1,0) {$n$};
    \draw (v) edge (w);
  \end{tikzpicture}
  \begin{tikzpicture}
    \node[int] (i) at (0,.5) {};
    \node (v1) at (-.5,-.2) {$\omega$};
    \node (v2) at (0,-.2) {$\omega$};
    \node (v3) at (.5,-.2) {$\omega$};
  \draw (i) edge (v1) edge (v2) edge (v3);
  \end{tikzpicture}
  \cdots 
  \begin{tikzpicture}
    \node[int] (i) at (0,.5) {};
    \node (v1) at (-.5,-.2) {$\omega$};
    \node (v2) at (0,-.2) {$\omega$};
    \node (v3) at (.5,-.2) {$\omega$};
  \draw (i) edge (v1) edge (v2) edge (v3);
  \end{tikzpicture}
  \begin{tikzpicture}
    \node[int] (i) at (0,.5) {};
    \node (v1) at (-.5,-.2) {$\epsilon$};
    \node (v2) at (0,-.2) {$\omega$};
    \node (v3) at (.5,-.2) {$\omega$};
  \draw (i) edge (v1) edge (v2) edge (v3);
  \end{tikzpicture}
\]
\[
    \Gamma^{(2)}_{\epsilon j}:=
    \begin{tikzpicture}
        \node (v) at (0,0) {$\epsilon$};
        \node (w) at (1,0) {$j$};
        \draw (v) edge (w);
    \end{tikzpicture}
  \begin{tikzpicture}
    \node (v) at (0,0) {$\omega$};
    \node (w) at (1,0) {$1$};
    \draw (v) edge (w);
  \end{tikzpicture}
  \cdots 
  \begin{tikzpicture}
    \node (v) at (0,0) {$\omega$};
    \node (w) at (1,0) {$n$};
    \draw (v) edge (w);
  \end{tikzpicture}
  \begin{tikzpicture}
    \node[int] (i) at (0,.5) {};
    \node (v1) at (-.5,-.2) {$\omega$};
    \node (v2) at (0,-.2) {$\omega$};
    \node (v3) at (.5,-.2) {$\omega$};
  \draw (i) edge (v1) edge (v2) edge (v3);
  \end{tikzpicture}
  \cdots 
  \begin{tikzpicture}
    \node[int] (i) at (0,.5) {};
    \node (v1) at (-.5,-.2) {$\omega$};
    \node (v2) at (0,-.2) {$\omega$};
    \node (v3) at (.5,-.2) {$\omega$};
  \draw (i) edge (v1) edge (v2) edge (v3);
  \end{tikzpicture}
\]
\[
    \Gamma^{(2)}_{ij}:=
  \begin{tikzpicture}
    \node (v) at (0,0) {$\omega$};
    \node (w) at (1,0) {$1$};
    \draw (v) edge (w);
  \end{tikzpicture}
  \cdots 
  \begin{tikzpicture}
    \node (v) at (0,0) {$\omega$};
    \node (w) at (1,0) {$n$};
    \draw (v) edge (w);
  \end{tikzpicture}
  \begin{tikzpicture}
    \node[int] (i) at (0,.5) {};
    \node (v1) at (-.5,-.2) {$\omega$};
    \node (v2) at (0,-.2) {$\omega$};
    \node (v3) at (.5,-.2) {$\omega$};
  \draw (i) edge (v1) edge (v2) edge (v3);
  \end{tikzpicture}
  \cdots 
  \begin{tikzpicture}
    \node[int] (i) at (0,.5) {};
    \node (v1) at (-.5,-.2) {$\omega$};
    \node (v2) at (0,-.2) {$\omega$};
    \node (v3) at (.5,-.2) {$\omega$};
  \draw (i) edge (v1) edge (v2) edge (v3);
  \end{tikzpicture}
  \begin{tikzpicture}
    \node[int] (i) at (0,.5) {};
    \node (v1) at (-.5,-.2) {$i$};
    \node (v2) at (0,-.2) {$j$};
    \node (v3) at (.5,-.2) {$\omega$};
  \draw (i) edge (v1) edge (v2) edge (v3);
  \end{tikzpicture}
\]
\[
  \Gamma^{(2)}_{\omega\epsilon\omega\epsilon}=
  \begin{tikzpicture}
    \node (v) at (0,0) {$\omega$};
    \node (w) at (1,0) {$1$};
    \draw (v) edge (w);
  \end{tikzpicture}
  \cdots 
  \begin{tikzpicture}
    \node (v) at (0,0) {$\omega$};
    \node (w) at (1,0) {$n$};
    \draw (v) edge (w);
  \end{tikzpicture}
  \begin{tikzpicture}
    \node[int] (i) at (0,.5) {};
    \node (v1) at (-.5,-.2) {$\omega$};
    \node (v2) at (0,-.2) {$\omega$};
    \node (v3) at (.5,-.2) {$\omega$};
  \draw (i) edge (v1) edge (v2) edge (v3);
  \end{tikzpicture}
  \cdots 
  \begin{tikzpicture}
    \node[int] (i) at (0,.5) {};
    \node (v1) at (-.5,-.2) {$\omega$};
    \node (v2) at (0,-.2) {$\omega$};
    \node (v3) at (.5,-.2) {$\omega$};
  \draw (i) edge (v1) edge (v2) edge (v3);
  \end{tikzpicture}
  \begin{tikzpicture}
    \node (v) at (0,0) {$\omega$};
    \node (w) at (1,0) {$\epsilon$};
    \draw (v) edge (w);
  \end{tikzpicture}
  \begin{tikzpicture}
    \node (v) at (0,0) {$\omega$};
    \node (w) at (1,0) {$\epsilon$};
    \draw (v) edge (w);
  \end{tikzpicture}
\]
\[
  \Gamma^{(2)}_{\omega\epsilon j}=
  \begin{tikzpicture}
    \node (v) at (0,0) {$\omega$};
    \node (w) at (1,0) {$1$};
    \draw (v) edge (w);
  \end{tikzpicture}
  \cdots 
  \begin{tikzpicture}
    \node (v) at (0,0) {$\omega$};
    \node (w) at (1,0) {$n$};
    \draw (v) edge (w);
  \end{tikzpicture}
  \begin{tikzpicture}
    \node[int] (i) at (0,.5) {};
    \node (v1) at (-.5,-.2) {$\omega$};
    \node (v2) at (0,-.2) {$\omega$};
    \node (v3) at (.5,-.2) {$\omega$};
  \draw (i) edge (v1) edge (v2) edge (v3);
  \end{tikzpicture}
  \cdots 
  \begin{tikzpicture}
    \node[int] (i) at (0,.5) {};
    \node (v1) at (-.5,-.2) {$\omega$};
    \node (v2) at (0,-.2) {$\omega$};
    \node (v3) at (.5,-.2) {$\omega$};
  \draw (i) edge (v1) edge (v2) edge (v3);
  \end{tikzpicture}
  \begin{tikzpicture}
    \node[int] (i) at (0,.5) {};
    \node (v1) at (-.5,-.2) {$j$};
    \node (v2) at (0,-.2) {$\omega$};
    \node (v3) at (.5,-.2) {$\omega$};
  \draw (i) edge (v1) edge (v2) edge (v3);
  \end{tikzpicture}
  \begin{tikzpicture}
    \node (v) at (0,0) {$\omega$};
    \node (w) at (1,0) {$\epsilon$};
    \draw (v) edge (w);
  \end{tikzpicture}
\]
\[
    \Gamma^{(2)}_{i;j}:=
  \begin{tikzpicture}
    \node (v) at (0,0) {$\omega$};
    \node (w) at (1,0) {$1$};
    \draw (v) edge (w);
  \end{tikzpicture}
  \cdots 
  \begin{tikzpicture}
    \node (v) at (0,0) {$\omega$};
    \node (w) at (1,0) {$n$};
    \draw (v) edge (w);
  \end{tikzpicture}
  \begin{tikzpicture}
    \node[int] (i) at (0,.5) {};
    \node (v1) at (-.5,-.2) {$\omega$};
    \node (v2) at (0,-.2) {$\omega$};
    \node (v3) at (.5,-.2) {$\omega$};
  \draw (i) edge (v1) edge (v2) edge (v3);
  \end{tikzpicture}
  \cdots 
  \begin{tikzpicture}
    \node[int] (i) at (0,.5) {};
    \node (v1) at (-.5,-.2) {$\omega$};
    \node (v2) at (0,-.2) {$\omega$};
    \node (v3) at (.5,-.2) {$\omega$};
  \draw (i) edge (v1) edge (v2) edge (v3);
  \end{tikzpicture}
  \begin{tikzpicture}
    \node[int] (i) at (0,.5) {};
    \node (v1) at (-.5,-.2) {$i$};
    \node (v2) at (0,-.2) {$\omega$};
    \node (v3) at (.5,-.2) {$\omega$};
  \draw (i) edge (v1) edge (v2) edge (v3);
  \end{tikzpicture}
  \begin{tikzpicture}
    \node[int] (i) at (0,.5) {};
    \node (v1) at (-.5,-.2) {$j$};
    \node (v2) at (0,-.2) {$\omega$};
    \node (v3) at (.5,-.2) {$\omega$};
  \draw (i) edge (v1) edge (v2) edge (v3);
  \end{tikzpicture}
\]
After killing terms involving trees with at least 4 leaves, the differential maps $\Gamma^{(2)}_\epsilon$, $\Gamma^{(2)}_{ij}$, and $\Gamma^{(2)}_{i;j}$ to $0$. On the remaining generators, it is given by:
\begin{align*}
\Gamma^{(0)} &\mapsto \sum_j \pm  \Gamma^{(2)}_{\epsilon j} + (const)\Gamma^{(2)}_\epsilon &
\Gamma^{(2)}_{\epsilon j} &\mapsto \sum_i \pm \Gamma^{(2)}_{ij} \\ 
\Gamma^{(2)}_{\omega\epsilon\omega\epsilon} &\mapsto \pm \Gamma^{(2)}_\epsilon + \sum_j \pm \Gamma^{(2)}_{\omega\epsilon j} &
\Gamma^{(2)}_{\omega\epsilon j}&\mapsto \sum_i \pm \Gamma^{(2)}_{i;j}. 
\end{align*}
Here one needs to take care that when $n = 1$ or $g = 1$, some of these generators are absent.  More precisely, the generators $\Gamma^{(2)}_{ij}$ and $\Gamma^{(2)}_{i;j}$ are not present when $n = 1$, nor are $\Gamma^{(2)}_{i;j}$, $\Gamma^{(2)}_{\omega\epsilon\omega\epsilon}$, $\Gamma^{(2)}_{\omega\epsilon j}$ when $g = 1$.  When all of the generators are present, the cohomology consists of one copy of the sign representation $V_{1^n}$ of $\bbS_n$, represented by $\Gamma^{(2)}_\epsilon+\cdots$, and two copies of the irreducible representation $V_{31^{n-3}}$, represented by linear combinations of graphs of the form $\Gamma^{(2)}_{ij}$ and $\Gamma^{(2)}_{i;j}$ respectively.
Taking into account the special cases $n=1$ and $g=1$, we arrive at the following: 
\begin{align*}
    H^k(\yB_{1,16}) &=
    \begin{cases}
        V_{31^{13}} & \text{for $k=16$} \\
        0 & \text{otherwise}
    \end{cases}
    &
    H^k(\yB_{3,13}) &=
    \begin{cases}
        V_{1^{13}} & \text{for $k=18$} \\
        V_{31^{10}}\oplus V_{31^{10}} & \text{for $k=19$} \\
        0 & \text{otherwise}
    \end{cases}
    \\
    H^k(\yB_{5,10}) &=
    \begin{cases}
        V_{1^10} & \text{for $k=21$} \\
        V_{31^{7}}\oplus V_{31^{7}} & \text{for $k=22$} \\
        0 & \text{otherwise}
    \end{cases}
    &
    H^k(\yB_{7,7}) &=
    \begin{cases}
        V_{1^7} & \text{for $k=24$} \\
        V_{31^4}\oplus V_{31^4} & \text{for $k=25$} \\
        0 & \text{otherwise}
    \end{cases}\\
    H^k(\yB_{9,4}) &=
    \begin{cases}
        V_{1^{4}} & \text{for $k=27$} \\
        V_{31}\oplus V_{31} & \text{for $k=28$} \\
        0 & \text{otherwise}
    \end{cases}
    &
    H^k(\yB_{11,1}) &=
    \begin{cases}
        V_{1} & \text{for $k=30$} \\
        0 & \text{otherwise}
    \end{cases}
   .
\end{align*}

\subsection{Case \texorpdfstring{$E'(g,n)=3$}{E'=3}}
In the case $E'(g,n)=3$ the computation is slightly more complex, due to the increase in the number of generators. Proceeding as in \cite[Section 4.4]{PayneWillwacher24} we obtain:

\begin{align*}
    H^k(\yB_{2,15}) &=
    \begin{cases}
        V_{41^{7}}\oplus V_{321^{6}}\oplus V_{31^{8}} & \text{for $k=18$} \\
        0 & \text{otherwise}
    \end{cases}
    \\
    H^k(\yB_{4,12}) &=
    \begin{cases}
        V_{21^{10}} & \text{for $k=20$} \\
        V_{41^{8}}\oplus V_{41^{8}}\oplus V_{321^{7}}\oplus V_{31^{9}} & \text{for $k=21$} \\
        0 & \text{otherwise}
    \end{cases}
    \\
    H^k(\yB_{6,9}) &=
    \begin{cases}
        V_{21^{7}} & \text{for $k=23$} \\
        V_{41^5}\oplus V_{41^5}\oplus V_{321^4}\oplus V_{31^{6}} & \text{for $k=24$} \\
        0 & \text{otherwise}
    \end{cases}
    \\
    H^k(\yB_{8,6}) &=
    \begin{cases}
        V_{21^{4}} & \text{for $k=26$} \\
        V_{41^5}\oplus V_{41^5}\oplus V_{321}\oplus V_{31^{3}} & \text{for $k=27$} \\
        0 & \text{otherwise}
    \end{cases}
    \\
    H^k(\yB_{10,3}) &=
    \begin{cases}
        V_{21} & \text{for $k=29$} \\
        0 & \text{otherwise}
    \end{cases}
    \\
    H^k(\yB_{12,0}) &=0
\end{align*}
\noindent The computation for $B^{15}_{12,0}$ completes the proof of Theorem~\ref{thm:s16}.

\section{Infinite families of cohomology classes in \texorpdfstring{$H^{(15,0)}(\M_{g})$}{H15(Mg)}} \label{sec:inj}
We now restrict attention to the special case of $n = 0$, allowing the genus $g$ to be arbitrarily large.  We abbreviate: 
\begin{align*}
B_g^{15} &:= B_{g,0}^{15},
&
C_g^{15} &:= C_{g,0}^{15}.
\end{align*}

Let $\mathsf{GC}_0$ be the graph complex generated by connected graphs without self-loops in which every vertex has valence at least $3$. Each generator is given an orientation by a total ordering of the edges. The cohomological degree of a generator is the number of edges. Isomorphic graphs are identified, and two orientations are identified up to sign. The differential acts by splitting vertices. Therefore, the differential preserves the loop order, and we write $\mathsf{GC}_0^{(g)}$ for the part of loop order $g$. 

Analogously to \cite[Theorems 5.1 and 6.1]{PayneWillwacher24} one shows the following theorems.
\begin{thm}\label{thm:emb1}
There is a map of complexes
\[
F \colon \Sym^{14}(\GC_0)^{(g-1)}[-29] \oplus \Sym^{14}(\GC_0)^{(g-2)}[-30] \to C_{g}^{15}
\]
that induces an injective map on the level of cohomology
\[
H^{k-29}(\Sym^{14}(\GC_0)^{(g-1)}) \oplus H^{k-30}(\Sym^{14}(\GC_0)^{(g-2)})
\to H^k(C_g^{15}).
\]
\end{thm}

\begin{proof}[Idea of proof.]
The result is shown by the same argument as \cite[Theorem 5.1]{PayneWillwacher24},  replacing the constant 11 by 15. 
We do not reproduce the full proof here, but instead indicate the main underlying ideas. 
The generators of $C_g^{15}$ can be represented by blown-up graphs so that 0--14 of the external legs are labeled by a symbol $\omega$, and the other external legs are labeled by a symbol $\epsilon$. The differential splits into three pieces $\delta=\delta_s^\bullet+\delta_s^\circ+\delta_\omega$, with $\delta_s^\bullet$ splitting a vertex and $\delta_s^\circ$ fusing some $\epsilon$-legs and at most one $\omega$-leg into a new vertex, while $\delta_\omega$ replaces one $\omega$ at a leg by $\epsilon$. 
We also write $\delta_s:=\delta_s^\bullet+\delta_s^\circ$.

Now consider a cocycle $x \in C_g^{15}$ and decompose it as 
\[
x=x_0+\cdots+x_{14},
\]
where $x_j$ is the part with $j$ many $\omega$-legs.
Then the cocycle condition $\delta x=0$ implies that $x_{14}$ is a cocycle in the complex $(C_g^{15}, \delta_s)$, i.e., $\delta_s x_{14}=0$.
Furthermore, if $x_{14}$ is not exact in $(C_g^{15}, \delta_s)$, then $x$ cannot be exact in $(C_g^{15}, \delta_s+\delta_\omega)$,
so we have constructed a non-trivial cohomology class as desired.

Next, suppose we have graph cocycles $\gamma_1,\dots,\gamma_{14}\in \GC_0$. Then we may build a cocycle 
\[
x_{14} = 
\begin{tikzpicture}
    \node[ext] (v1) at (0,.3) {$\scriptstyle \gamma_1$}; 
    \node (w1) at (0,-.3) {$\scriptstyle \omega$};
    \draw (v1) edge (w1);
    \node at (.5,0) {$\cdots$};
    \node[ext] (vn) at (1,.3) {$\scriptstyle \gamma_{14}$}; 
    \node (wn) at (1,-.3) {$\scriptstyle \omega$};
    \draw (vn) edge (wn);
\end{tikzpicture}
\]
by attaching $\omega$-legs to $\gamma_1,\dots,\gamma_{14}$ and taking the union of the resulting graphs.
This assignment produces a morphism of complexes $\Sym^{14}(\GC_0)^{(g-1)}[-29] \to (C_{g}^{15},\delta_s)$.
It is shown in \cite{PayneWillwacher24} that this assignment has the following properties.
\begin{enumerate}
    \item It is an injection on homology. This can essentially be deduced by an earlier result of Turchin and the fourth author \cite{TurchinWillwacher}, see also the Appendix of \cite{PayneWillwacher24}.
    \item It extends to a morphism of complexes $\Sym^{14}(\GC_0)^{(g-1)}[-29] \to (C_{g}^{15},\delta)$. The extension involves explicit combinatorial expressions (that are tedious to define) for the missing pieces $x_{13},x_{12},\dots$ of our cocycle $x$ above.
\end{enumerate}
By our consideration above, this shows that we have an embedding in cohomology
$$\Sym^{14}(H(\GC_0))^{(g-1)}[-29] \to H(C_{g}^{15},\delta).$$
The embedding of the second summand $\Sym^{14}(H(\GC_0))^{(g-2)}[-30]$ is shown similarly, except that one adds an additional $\epsilon$-$\epsilon$-edge (in the appropriate way) to the leading order term $x_{14}$, increasing the genus and degree by one.
\end{proof}

\begin{thm}\label{thm:emb 2}
There is a 
 map of dg vector spaces 
\begin{multline*}
E \colon 
\Sym^{13}(\GC_0)^{(g-3)}[-30]
\oplus
\Sym^{10}(\GC_0)^{(g-5)}[-30]
\oplus
\Sym^{7}(\GC_0)^{(g-7)}[-30] \\
\oplus
\Sym^{4}(\GC_0)^{(g-9)}[-30]
 \oplus
\GC_0^{(g-11)}[-30]
\to
B_g^{15}
\end{multline*}
that gives rise to an injective 
map on cohomology
\begin{multline*}
    E\colon 
H^{k-30}(\Sym^{13}(\GC_0))^{(g-3)} 
\oplus
H^{k-30}(\Sym^{10}(\GC_0))^{(g-5)} 
\oplus
H^{k-30}(\Sym^{7}(\GC_0))^{(g-7)} 
\\
\oplus
H^{k-30}(\Sym^{4}(\GC_0))^{(g-9)} 
\oplus
H^{k-30}(\GC_0)^{(g-11)} 
\to H^k(B_g).
\end{multline*}
\end{thm}
\begin{proof}[Idea of proof.]
    The result is shown by the same argument as \cite[Theorem 6.1]{PayneWillwacher24}, replacing the constant 11 by 15.
    Hence we again only sketch the main idea, which is as follows.
    Suppose we are given (non-exact) graph cocycles $\gamma_1,\dots,\gamma_r$ with $r$ any number in $\{1,4,7,10,13\}$. 
    Then we may build a cochain $y_{16} \in B_g^{15}$ with $16$ $\omega$-legs of the form
    \[
y_{16} = 
    \underbrace{
    \begin{tikzpicture}
      \node[int] (i) at (0,.5) {};
      \node (v1) at (-.5,-.2) {$\omega$};
      \node (v2) at (0,-.2) {$\omega$};
      \node (v3) at (.5,-.2) {$\omega$};
    \draw (i) edge (v1) edge (v2) edge (v3);
    \end{tikzpicture} 
    \cdots 
    \begin{tikzpicture}
      \node[int] (i) at (0,.5) {};
      \node (v1) at (-.5,-.2) {$\omega$};
      \node (v2) at (0,-.2) {$\omega$};
      \node (v3) at (.5,-.2) {$\omega$};
    \draw (i) edge (v1) edge (v2) edge (v3);
    \end{tikzpicture}  
    }_{\frac{16-r}3 \times }
    \begin{tikzpicture}
      \node[ext] (i) at (0,.5) {$\gamma_1$};
      \node (v1) at (0,-.2) {$\omega$};
    \draw (i) edge (v1);
    \end{tikzpicture}
    \cdots 
    \begin{tikzpicture}
      \node[ext] (i) at (0,.5) {$\gamma_r$};
      \node (v1) at (0,-.2) {$\omega$};
    \draw (i) edge (v1);
    \end{tikzpicture}.
\]
Since $y_{16}$ has precisely 16 $\omega$-legs, this restricts the possible $r$ to the set $\{1,4,7,10,13\}$.
The element $y_{16}$ satisfies $\delta_s y_{16}=0$, but generally $\delta_\omega y_{16}\neq 0$.
However, as seen in \cite{PayneWillwacher24}, one may define explicitly a cochain $y_{15}\in B_g^{15}$ with $15$ $\omega$-legs such that $(\delta_s+\delta_\omega)(y_{15}+y_{16})=0$, so that $y=y_{15}+y_{16}$ is a cocycle in $(B_g^{15},\delta)$. This construction defines the map $E$ of the theorem.

To check the theorem, one (essentially) has to verify that our cocycle  $y$ is not exact. To this end we may consider its image $x:=\delta_\omega y\in C_g^{15}$ in the complex $C_g^{15}$ (see \eqref{equ:HB HC iso}, \eqref{equ:B C short exact}), and then check that $x$ defines a non-trivial cohomology class in the same way as we did in the proof of Theorem \ref{thm:emb1}.
In fact, it turns out that $x$ has the form
\[
  \begin{tikzpicture}
    \node (v) at (0,0) {$\omega$};
    \node (w) at (1,0) {$\epsilon$};
    \draw (v) edge (w);
  \end{tikzpicture}
  \begin{tikzpicture}
    \node (v) at (0,0) {$\epsilon$};
    \node (w) at (1,0) {$\epsilon$};
    \draw (v) edge (w);
  \end{tikzpicture}
  \begin{tikzpicture}
    \node[int] (i) at (0,.5) {};
    \node (v1) at (-.5,-.2) {$\omega$};
    \node (v2) at (0,-.2) {$\omega$};
    \node (v3) at (.5,-.2) {$\omega$};
  \draw (i) edge (v1) edge (v2) edge (v3);
  \end{tikzpicture}
  \cdots 
  \begin{tikzpicture}
    \node[int] (i) at (0,.5) {};
    \node (v1) at (-.5,-.2) {$\omega$};
    \node (v2) at (0,-.2) {$\omega$};
    \node (v3) at (.5,-.2) {$\omega$};
  \draw (i) edge (v1) edge (v2) edge (v3);
  \end{tikzpicture}  
  \begin{tikzpicture}
    \node[ext] (i) at (0,.5) {$\gamma_1$};
    \node (v1) at (0,-.2) {$\omega$};
  \draw (i) edge (v1);
  \end{tikzpicture}
  \cdots 
  \begin{tikzpicture}
    \node[ext] (i) at (0,.5) {$\gamma_r$};
    \node (v1) at (0,-.2) {$\omega$};
  \draw (i) edge (v1);
  \end{tikzpicture}
+(\cdots),
\]
where $(\cdots)$ are exact terms or terms with fewer connected components.
The leading order term displayed above represents a non-trivial cohomology class in $(C_g^{15},\delta_s)$, just as in  \cite[Proposition 5.9]{PayneWillwacher24}, and hence $x$ represents a non-trivial cohomology class in $(C_g^{15},\delta)$.
\end{proof}

One furthermore verifies that the images of the injections of Theorems \ref{thm:emb1} and \ref{thm:emb 2} are linearly independent inside $\gr^W_{15}H^k_c(\M_g)$, just as in \cite[Section 6.4]{PayneWillwacher24}.
\begin{cor}\label{cor:emb1}
There is an injection 
\begin{multline*}
\big(
H^{k-29}(\Sym^{14}(\GC_0)^{(g-1)}) \oplus H^{k-30}(\Sym^{14}(\GC_0)^{(g-2)})
\oplus
H^{k-30}(\Sym^{13}(\GC_0))^{(g-3)} 
\\
\oplus
H^{k-30}(\Sym^{10}(\GC_0))^{(g-5)} 
\oplus
H^{k-30}(\Sym^{7}(\GC_0))^{(g-7)} 
\oplus
H^{k-30}(\Sym^{4}(\GC_0))^{(g-9)} 
\\\oplus
H^{k-30}(\GC_0)^{(g-11)} 
\big) \otimes \s_{16} 
\to \gr^W_{15}H^k_c(\M_g).
\end{multline*}
\end{cor}

\begin{cor}\label{cor:exp growth}
    The dimension of the compactly supported cohomology $H_c^{2g+k}(\M_g)$ grows at least exponentially with $g$ for all $k\leq 73$ with the possible exception of $k=1,4,7,71$. 
\end{cor}
\noindent This corollary was previously known for many, but not all, values of $k$ in the stated range; see \cite[Corollary 1.2]{PayneWillwacher24}. The $21$ new values of $k$ are:
\[
20,51, 54,55,56,\dots,69,70, 72, 73.
\]
\begin{proof}[Proof of Corollary \ref{cor:exp growth}]
It is known that the dimension of $H^{2g+\kappa}(\GC_0)^{(g)}$ grows at least exponentially with $g$ for $\kappa=0$ and $\kappa=3$, by \cite{CGP21} and \cite[Proposition 2.4]{PayneWillwacher21}, respectively.
Furthermore, $H^{27}(\GC_0)^{(10)}\cong \Q$ by computer calculations.
Suppose we have an injection 
\[
H^{k-\delta}(\Sym^p(\GC_0))^{g-\delta'} \to H^k_c(\M_g),
\]
as asserted in the previous theorems for various values of $p$, $\delta$, and $\delta'$.
Then $H^{2g+k}_c(\M_g)$ grows at least exponentially with $g$ for any $k$ of the form 
\[
\delta - 2\delta' + \kappa_1 +\cdots +\kappa_p,
\]
where the $\kappa_j\in \{0,3,7\}$ are such that at least one of them is in $\{0,3\}$ and at most one of them is $7$.
The latter condition arises because the known cohomology in $H^{2g+7}(\GC_0)^{(g)}$, namely $H^{27}(\GC_0)^{(10)} \cong \Q$, is one-dimensional and concentrated in odd degree. Hence its higher symmetric powers vanish.
If $p \geq 2$, the resulting values of $k$ are thus
\[
\{ \delta -2\delta' + 3j\mid j=0,\dots,p \}
\cup 
\{ \delta -2\delta' + 3j+7 \mid j=0,\dots,p-1 \}.
\]
For example, consider the inclusion 
\[
H^{k-29}(\Sym^{14}(\GC_0))^{(g-1)}\to H_c^k(\M_g)
\]
from Corollary \ref{cor:emb1}. This provides at least exponential growth for $k$ in the set 
\[
\{27,30,33,\dots, 51, 54, 57, 60, 63, 66, 69 \} \cup \{34,37,40,\dots ,55,58,61,64,67,70,73 \}.
\]
Similarly, the other terms of Corollary \ref{cor:emb1} give at least exponential growth for $k$ in the set:
\begin{align*}
&\{26,29,\dots, 68 \} \ \cup \ \{33,36,\dots,72 \} \  
\cup \ \{24,27,\dots,63\} \ \cup \ \{31,34,\dots,67\} \\ 
\cup \ & \{20,23,\dots,50\} \ \cup \ \{27,30,\dots,54\} \
\cup \   \{16,19,\dots,37\} \ \cup \ \{23,26,\dots,41\} \\
\cup \ & \{12,15,\dots,24\} \   \cup \  \{17,20,\dots,26\} \ 
\cup \ \{ 8,11\}.
\end{align*}
This then yields the list of Corollary \ref{cor:exp growth}.
\end{proof}

\section{Euler characteristic of \texorpdfstring{$\gr_{15,0}H^*_c(\M_{g,n})$}{Gr150}}\label{sec:euler}

The weight 11 Euler characteristic computation of \cite[Section 7]{PayneWillwacher24}, carries over to the type $(15,0)$, by replacing the truncation constant 10 with 14 and hence using $C_{g,n}^{15}$ in place of the complex $C_{g,n}$ in loc. cit. We just state the final result here.

\begin{thm} \label{thm:chi-S16}
    The equivariant Euler characteristic of the type $(15,0)$ compactly supported cohomology of the moduli space of curves is computed by the following generating function:
\begin{multline*} 
\sum_{g,n\geq 0; \, 2g+n\geq 3} u^{g+n} \chi_{\bbS_n}(\gr_{15,0}H_c(\M_{g,n}))
  =
  \sum_{g,n}
  u^{g+n}
  \chi_{\bbS_n}(C^{15}_{g,n} )
  \\=
  -u\  T_{\leq 14}\bigg(
  \prod_{\ell\geq 1} 
  \frac { 
    U_\ell(\frac 1 \ell \sum_{d\mid \ell} \mu(\ell/d) (-p_d  +1-w^d ), u )
  }
  { 
    U_\ell(\frac 1 \ell \sum_{d\mid \ell} \mu(\ell/d) (-p_d), u )
  }-1\bigg)\, ,
\end{multline*}
where $T_{\leq 14}\left(\sum_{j=0}^\infty a_j w^j\right):=\sum_{j=0}^{14} a_j$,
\begin{align*}
    E_\ell&:= \frac 1 \ell \sum_{d\mid \ell}\mu(\ell/d)\frac 1 {u^d},
    &
    \lambda_\ell &:= u^\ell (1-u^\ell)\ell,
    &
    B(z) &:= \sum_{r\geq 2}\frac{B_r}{r(r-1)} \frac 1 {z^{r-1}}
\end{align*}
and $U_\ell(X,u)$ is characterized by 
\begin{align*}
    \log U_\ell(X,u)
   &=   
    \log \frac {(-\lambda_\ell)^X \Gamma(-E_\ell+X) }{\Gamma(-E_\ell)}
    \\&=
        X\left(\log(\lambda_\ell E_\ell)-1 \right)+(-E_\ell+X-\textstyle{\frac 1 2} )\log(1-\textstyle{\frac X{E_\ell}}) + B(-E_\ell+X)- B(-E_\ell).
\end{align*}
\end{thm}

We provide a computer generated table for small values of $(g,n)$ in Figure~\ref{fig:ec}.

\begin{rem}
In \cite{CLPW}, we studied the asymptotic behavior of the weight 11 Euler characteristic $\chi(\gr_{11}H_c(\M_{g}))$ as $g\to \infty$.
Given that the formula for the type $(15,0)$ Euler characteristic above is almost identical to the type $(11,0)$ case (modulo the replacement of the constant 10 by 14 in the truncation operator $T_{\leq k}$), the analysis of loc. cit. also applies to the case above.
We hence find the asymptotic formula as $g\to \infty$:
\[
\chi(\gr_{15,0}H_c(\M_{g}))
\sim 
\begin{cases}
 D_\infty^{ev} \frac{(-1)^{g/2} (g-2)! }{(2\pi)^g } & \text{for $g$ even} \\
 D_\infty^{odd} \frac{(-1)^{(g-1)/2} (g-2)! }{(2\pi)^{g} }
     & \text{for $g$ odd} 
\end{cases}
\]
with the constants
\begin{align*}
    D_\infty^{ev}
    &=
-
\sum_{j=8}^\infty \frac{(-4\pi^2)^{j} }{j(2j-15)! 14!}
   \approx 
   0.498203, \mbox{ and}
\\
D_{\infty}^{odd}
    &=
    -
\sum_{j=7}^\infty \frac{4\pi (-4\pi^2)^{j} }{(2j+1)(2j-14)! 14!}
    \approx 
    1.24975.
\end{align*}

\end{rem}

\begin{figure}
    \centering
 \begin{adjustbox}{angle=90, scale=.7}
\begin{tabular}{|g|m{1cm}|m{1cm}|m{3cm}|m{4cm}|m{4cm}|m{5cm}|m{6cm}|} 
\hline \rowcolor{Gray} $g$,$n$ & 0 & 1 & 2 & 3 & 4 & 5 & 6\\  
\hline
 7 & $ 0 $ & $ 0 $ & $ 0 $ & $ 0 $ & $ 0 $ & $ 0 $ & $ s_{1,1,1,1,1,1} $ \\  
\hline
 8 & $ 0 $ & $ 0 $ & $ 0 $ & $ 0 $ & $ 0 $ & $ s_{2,1,1,1} $ & $ s_{2,1,1,1,1} - s_{3,1,1,1} - s_{3,2,1} - 2s_{4,1,1} $ \\  
\hline
 9 & $ 0 $ & $ 0 $ & $ 0 $ & $ -s_{1,1,1} $ & $ -s_{1,1,1,1} + 2s_{3,1} $ & $ -2s_{2,2,1} + 2s_{3,1,1} - 2s_{3,2} - s_{4,1} - 3s_{5} $ & $ s_{1,1,1,1,1,1} + 7s_{2,1,1,1,1} + s_{2,2,1,1} + 4s_{2,2,2} + 6s_{3,1,1,1} + 3s_{3,2,1} - 2s_{3,3} + s_{4,1,1} + 3s_{4,2} - 3s_{5,1} $ \\  
\hline
 10 & $ 0 $ & $ 0 $ & $ -s_{2} $ & $ -s_{2,1} $ & $ -4s_{1,1,1,1} - 2s_{2,1,1} + s_{3,1} + 3s_{4} $ & $ -10s_{1,1,1,1,1} - s_{2,1,1,1} - s_{2,2,1} + 14s_{3,1,1} + s_{3,2} + 4s_{4,1} $ & $ -17s_{1,1,1,1,1,1} - 2s_{2,1,1,1,1} - 11s_{2,2,1,1} - 8s_{2,2,2} + 33s_{3,1,1,1} - 10s_{3,2,1} - 19s_{3,3} + 3s_{4,1,1} - 20s_{4,2} - 24s_{5,1} - 8s_{6} $ \\  
\hline
 11 & $ s_{} $ & $ s_{1} $ & $ 0 $ & $ -4s_{2,1} - 2s_{3} $ & $ -2s_{1,1,1,1} - 14s_{2,1,1} + 3s_{2,2} + 8s_{4} $ & $ -11s_{1,1,1,1,1} - 31s_{2,1,1,1} - 9s_{2,2,1} + 4s_{3,1,1} + 19s_{3,2} + 35s_{4,1} + 12s_{5} $ & $ -44s_{1,1,1,1,1,1} - 58s_{2,1,1,1,1} - 20s_{2,2,1,1} - 58s_{2,2,2} + 69s_{3,1,1,1} + 21s_{3,2,1} - 11s_{3,3} + 111s_{4,1,1} - 44s_{4,2} - 19s_{5,1} - 51s_{6} $ \\  
\hline
 12 & $ 0 $ & $ 2s_{1} $ & $ 6s_{1,1} + 2s_{2} $ & $ 12s_{1,1,1} - 2s_{2,1} - 16s_{3} $ & $ 16s_{1,1,1,1} - 10s_{2,1,1} + 8s_{2,2} - 40s_{3,1} - 10s_{4} $ & $ 7s_{1,1,1,1,1} - 89s_{2,1,1,1} + 5s_{2,2,1} - 127s_{3,1,1} + 67s_{3,2} + 48s_{4,1} + 81s_{5} $ & $ -36s_{1,1,1,1,1,1} - 236s_{2,1,1,1,1} + 5s_{2,2,1,1} - 123s_{2,2,2} - 127s_{3,1,1,1} + 130s_{3,2,1} + 203s_{3,3} + 266s_{4,1,1} + 119s_{4,2} + 223s_{5,1} + 25s_{6} $ \\  
\hline
 13 & $ -2s_{} $ & $ -2s_{1} $ & $ 2s_{1,1} + 12s_{2} $ & $ 26s_{1,1,1} + 20s_{2,1} $ & $ 90s_{1,1,1,1} + 56s_{2,1,1} - 41s_{2,2} - 117s_{3,1} - 95s_{4} $ & $ 213s_{1,1,1,1,1} + 28s_{2,1,1,1} + 95s_{2,2,1} - 392s_{3,1,1} + 6s_{3,2} - 123s_{4,1} + 91s_{5} $ & $ 324s_{1,1,1,1,1,1} - 194s_{2,1,1,1,1} - 105s_{2,2,1,1} + 156s_{2,2,2} - 1451s_{3,1,1,1} - 129s_{3,2,1} + 580s_{3,3} - 497s_{4,1,1} + 756s_{4,2} + 924s_{5,1} + 495s_{6} $ \\  
\hline
 14 & $ -s_{} $ & $ -6s_{1} $ & $ -7s_{1,1} + 20s_{2} $ & $ -s_{1,1,1} + 104s_{2,1} + 84s_{3} $ & $ 92s_{1,1,1,1} + 334s_{2,1,1} - 87s_{2,2} + 13s_{3,1} - 226s_{4} $ & $ 381s_{1,1,1,1,1} + 816s_{2,1,1,1} + 170s_{2,2,1} - 131s_{3,1,1} - 742s_{3,2} - 1135s_{4,1} - 558s_{5} $ & $ 1174s_{1,1,1,1,1,1} + 1380s_{2,1,1,1,1} + 235s_{2,2,1,1} + 1814s_{2,2,2} - 2494s_{3,1,1,1} - 868s_{3,2,1} + 232s_{3,3} - 3535s_{4,1,1} + 1360s_{4,2} + 575s_{5,1} + 1393s_{6} $ \\  
\hline
 15 & $ -2s_{} $ & $ -29s_{1} $ & $ -78s_{1,1} - 41s_{2} $ & $ -196s_{1,1,1} + 61s_{2,1} + 290s_{3} $ & $ -234s_{1,1,1,1} + 410s_{2,1,1} + 55s_{2,2} + 1001s_{3,1} + 375s_{4} $ & $ -111s_{1,1,1,1,1} + 2215s_{2,1,1,1} - 233s_{2,2,1} + 2540s_{3,1,1} - 2169s_{3,2} - 1822s_{4,1} - 2071s_{5} $ & $ 1154s_{1,1,1,1,1,1} + 6337s_{2,1,1,1,1} + 2410s_{2,2,1,1} + 3359s_{2,2,2} + 4687s_{3,1,1,1} - 1410s_{3,2,1} - 4684s_{3,3} - 5607s_{4,1,1} - 3851s_{4,2} - 6258s_{5,1} - 1153s_{6} $ \\  
\hline
 16 & $ 14s_{} $ & $ 17s_{1} $ & $ -76s_{1,1} - 200s_{2} $ & $ -489s_{1,1,1} - 354s_{2,1} + 23s_{3} $ & $ -1553s_{1,1,1,1} - 1156s_{2,1,1} + 959s_{2,2} + 2477s_{3,1} + 2068s_{4} $ & $ -3169s_{1,1,1,1,1} + 435s_{2,1,1,1} - 369s_{2,2,1} + 8767s_{3,1,1} + 1650s_{3,2} + 4080s_{4,1} - 1118s_{5} $ & $ -5124s_{1,1,1,1,1,1} + 5055s_{2,1,1,1,1} + 3006s_{2,2,1,1} - 5495s_{2,2,2} + 28447s_{3,1,1,1} - 3461s_{3,2,1} - 15299s_{3,3} + 7079s_{4,1,1} - 25284s_{4,2} - 25064s_{5,1} - 13800s_{6} $ \\  
\hline
 17 & $ 2s_{} $ & $ 75s_{1} $ & $ 90s_{1,1} - 329s_{2} $ & $ -22s_{1,1,1} - 1782s_{2,1} - 1555s_{3} $ & $ -2000s_{1,1,1,1} - 5643s_{2,1,1} + 964s_{2,2} - 420s_{3,1} + 3326s_{4} $ & $ -6876s_{1,1,1,1,1} - 13779s_{2,1,1,1} - 1133s_{2,2,1} + 4966s_{3,1,1} + 17718s_{3,2} + 24375s_{4,1} + 12433s_{5} $ & $ -20444s_{1,1,1,1,1,1} - 24332s_{2,1,1,1,1} - 5940s_{2,2,1,1} - 33861s_{2,2,2} + 46884s_{3,1,1,1} + 12667s_{3,2,1} - 6374s_{3,3} + 63692s_{4,1,1} - 26319s_{4,2} - 12311s_{5,1} - 25091s_{6} $ \\  
\hline
 18 & $ 17s_{} $ & $ 303s_{1} $ & $ 900s_{1,1} + 642s_{2} $ & $ 2237s_{1,1,1} - 1512s_{2,1} - 3984s_{3} $ & $ 2669s_{1,1,1,1} - 6688s_{2,1,1} - 2284s_{2,2} - 15341s_{3,1} - 6084s_{4} $ & $ -2793s_{1,1,1,1,1} - 44027s_{2,1,1,1} - 6683s_{2,2,1} - 45860s_{3,1,1} + 28559s_{3,2} + 28180s_{4,1} + 32906s_{5} $ & $ -21411s_{1,1,1,1,1,1} - 89263s_{2,1,1,1,1} - 15165s_{2,2,1,1} - 27777s_{2,2,2} - 28229s_{3,1,1,1} + 116405s_{3,2,1} + 108382s_{3,3} + 159111s_{4,1,1} + 146742s_{4,2} + 155495s_{5,1} + 35908s_{6} $ \\  
\hline
\end{tabular}
 \end{adjustbox}
    \caption{The $\bbS_n$-equivariant Euler characteristic of 
$\gr_{15,0}^WH_c^* (\M_{g,n})$.}
    \label{fig:ec}
\end{figure}

\bibliographystyle{amsplain}
\bibliography{refs}

\providecommand{\bysame}{\leavevmode\hbox to3em{\hrulefill}\thinspace}
\providecommand{\MR}{\relax\ifhmode\unskip\space\fi MR }
\providecommand{\MRhref}[2]{%
  \href{http://www.ams.org/mathscinet-getitem?mr=#1}{#2}
}
\providecommand{\href}[2]{#2}
\begin{thebibliography}{10}

\bibitem{ArbarelloCornalba}
Enrico Arbarello and Maurizio Cornalba, \emph{Calculating cohomology groups of
  moduli spaces of curves via algebraic geometry}, Inst. Hautes \'{E}tudes Sci.
  Publ. Math. (1998), no.~88, 97--127 (1999). \MR{1733327}

\bibitem{BergstromData}
Jonas Bergstr{\"o}m, \emph{Cohomology of moduli spaces of curves},
  \url{https://github.com/jonasbergstroem/Cohomology-of-moduli-spaces-of-curves}.

\bibitem{BergstromFaberPayne}
Jonas Bergstr\"{o}m, Carel Faber, and Sam Payne, \emph{Polynomial point counts
  and odd cohomology vanishing on moduli spaces of stable curves}, Ann. of
  Math. (2) \textbf{199} (2024), no.~3, 1323--1365. \MR{4740541}

\bibitem{BCGP}
Francis Brown, Melody Chan, S{\o}ren Galatius, and Sam Payne, \emph{Hopf
  algebras in the cohomology of $\mathcal{A}_g$, $\mathrm{GL}_n(\mathbb{Z})$,
  and $\mathrm{SL}_n(\mathbb{Z})$}, preprint arXiv:2405.11528, 2024.

\bibitem{CanningLarsonPayne}
Samir Canning, Hannah Larson, and Sam Payne, \emph{The eleventh cohomology
  group of {$\overline{\mathcal{M}}_{g,n}$}}, Forum Math. Sigma \textbf{11}
  (2023), no.~Paper No. e62, 18 pp.

\bibitem{CLP-STE}
\bysame, \emph{Extensions of tautological rings and motivic structures in the
  cohomology of {$\overline{\mathcal{M}}_{g,n}$}}, Forum Math. Pi \textbf{12}
  (2024), Paper No. e23, 31 pp.

\bibitem{CLPW}
Samir Canning, Hannah Larson, Sam Payne, and Thomas Willwacher, \emph{Moduli
  spaces of curves with polynomial point counts}, preprint, arXiv:2410.19913,
  2024.

\bibitem{CGP21}
Melody Chan, S{\o}ren Galatius, and Sam Payne, \emph{Tropical curves, graph
  complexes, and top weight cohomology of $\mathcal{M}_g$}, J. Amer. Math. Soc.
  \textbf{34} (2021), no.~2, 565--594. \MR{4280867}

\bibitem{CGP22}
\bysame, \emph{Topology of moduli spaces of tropical curves with marked
  points}, Facets of algebraic geometry. {V}ol. {I}, London Math. Soc. Lecture
  Note Ser., vol. 472, Cambridge Univ. Press, Cambridge, 2022, pp.~77--131.
  \MR{4381898}

\bibitem{ChenevierLannes}
Ga\"{e}tan Chenevier and Jean Lannes, \emph{Automorphic forms and even
  unimodular lattices}, Ergebnisse der Mathematik und ihrer Grenzgebiete. 3.
  Folge. A Series of Modern Surveys in Mathematics, vol.~69, Springer, Cham,
  2019. \MR{3929692}

\bibitem{FPhandbook}
Carel Faber and Rahul Pandharipande, \emph{Tautological and non-tautological
  cohomology of the moduli space of curves}, Handbook of moduli. {V}ol. {I},
  Adv. Lect. Math. (ALM), vol.~24, Int. Press, Somerville, MA, 2013,
  pp.~293--330. \MR{3184167}

\bibitem{Getzler}
Ezra Getzler, \emph{The semi-classical approximation for modular operads},
  Comm. Math. Phys. \textbf{194} (1998), no.~2, 481--492. \MR{1627677}

\bibitem{Petersenappendix}
Rahul Pandharipande, Dmitri Zvonkine, and Dan Petersen, \emph{Cohomological
  field theories with non-tautological classes}, Ark. Mat. \textbf{57} (2019),
  no.~1, 191--213. \MR{3951280}

\bibitem{PayneWillwacher21}
Sam Payne and Thomas Willwacher, \emph{The weight two compactly supported
  cohomology of moduli spaces of curves}, To appear in Duke Math. J.
  arXiv:2110.05711v1, 2021.

\bibitem{PayneWillwacher24}
\bysame, \emph{Weight 11 {C}ompactly {S}upported {C}ohomology of {M}oduli
  {S}paces of {C}urves}, Int. Math. Res. Not. IMRN (2024), no.~8, 7060--7098.
  \MR{4735654}

\bibitem{Petersenlocalsystems}
Dan Petersen, \emph{Cohomology of local systems on the moduli of principally
  polarized abelian surfaces}, Pacific J. Math. \textbf{275} (2015), no.~1,
  39--61. \MR{3336928}

\bibitem{TurchinWillwacher}
Victor Turchin and Thomas Willwacher, \emph{On the homotopy type of the spaces
  of spherical knots in {$\mathbb{R}^n$}}, M\"unster J. Math. \textbf{14}
  (2021), no.~2, 537--558. \MR{4359842}

\end{thebibliography}
\end{document}